\documentclass{amsart}
\usepackage{amsmath,amssymb,amsthm,mathrsfs}

\theoremstyle{plain}
\numberwithin{equation}{section}
\theoremstyle{plain}
\newtheorem{theorem}{Theorem}[section]
\newtheorem{corollary}[theorem]{Corollary}
\newtheorem{lemma}[theorem]{Lemma}
\newtheorem{proposition}[theorem]{Proposition}
\theoremstyle{definition}

\newtheorem{example}[theorem]{Example}
\newtheorem{remark}[theorem]{Remark}

\newcommand{\ands}{\quad\mbox{and}\quad}
\newcommand{\kernel}{{\mathrm{Ker}}}
\newcommand{\Dom}{{\mathrm{Dom}}}
\newcommand{\codim}{{\mathrm{codim}\, }}
\newcommand{\Index}{{\mathrm{Index}}}
\newcommand{\Rat}{{\mathrm{Rat}}}
\newcommand{\Ran}{{\mathrm{Ran}}}

\newcommand{\cP}{{\mathcal P}}
\newcommand{\cQ}{{\mathcal Q}}
\newcommand{\cW}{{\mathcal W}}
\newcommand{\cR}{{\mathcal R}}
\newcommand{\BC}{{\mathbb C}}
\newcommand{\BT}{{\mathbb T}}
\newcommand{\BD}{{\mathbb D}}

\newcommand{\BP}{{\mathbb P}}
\newcommand{\BZ}{{\mathbb Z}}
\newcommand{\wtil}[1]{{\widetilde{#1}}}
\newcommand{\what}[1]{{\widehat{#1}}}

\newcommand{\al}{\alpha}

\newcommand{\om}{\omega}

\begin{document}

\title[A Toeplitz-like operator with rational symbol having poles on the unit circle I]{A Toeplitz-like operator with rational symbol having poles on the unit circle I: Fredholm properties}

\author[G.J. Groenewald]{G.J. Groenewald}
\address{G.J. Groenewald, Department of Mathematics, Unit for BMI, North-West
University,
Potchefstroom, 2531 South Africa}
\email{Gilbert.Groenewald@nwu.ac.za}

\author[S. ter Horst]{S. ter Horst}
\address{S. ter Horst, Department of Mathematics, Unit for BMI, North-West
University, Potchefstroom, 2531 South Africa}
\email{Sanne.TerHorst@nwu.ac.za}

\author[J. Jaftha]{J. Jaftha}
\address{J. Jaftha, Numeracy Center, University of Cape Town, Rondebosch 7701; Cape Town; South Africa}
\email{Jacob.Jaftha@uct.ac.za}

\author[A.C.M. Ran]{A.C.M. Ran}
\address{A.C.M. Ran, Department of Mathematics, Faculty of Science, VU Amsterdam, De Boelelaan 1081a, 1081 HV Amsterdam, The Netherlands and Unit for BMI, North-West~University, Potchefstroom, South Africa}
\email{a.c.m.ran@vu.nl}

\thanks{This work is based on the research supported in part by the National Research Foundation of South Africa (Grant Number 90670 and 93406).\\
Part of the research was done during a sabbatical of the third author, in which time several research visits to VU Amsterdam and North-West University were made. Support from University of Cape Town and the Department of Mathematics, VU Amsterdam is gratefully acknowledged.
}

\begin{abstract}
In this paper a definition is given for an unbounded Toeplitz-like operator with rational symbol which has poles on the unit circle. It is shown that the operator is Fredholm if and only if the symbol has no zeroes on the unit circle, and a formula for the index is given as well. Finally, a matrix representation of the operator is discussed.
\end{abstract}

\subjclass[2010]{Primary 47B35, 47A53; Secondary 47A68}

\keywords{Toeplitz operators, unbounded operators, Fredholm properties}

\maketitle


\section{Introduction}\label{S:Intro}


The Toeplitz operator $T_\om$ on $H^p=H^p(\BD)$, $1 < p<\infty$, over the unit disc $\BD$ with rational symbol $\om$ having no poles on the unit circle $\BT$ is the bounded linear operator defined by
\[
T_\omega:H^p\to H^p,\quad T_\omega f = \BP \om f \ (f\in H^p),
\]
with $\BP$ the Riesz projection of $L^p=L^p(\BT)$ onto $H^p$. This operator, and many of its variations, has been extensively studied in the literature,  cf., \cite{BS06,DG02,GGK90,V08} and the references given there.

In this paper the case where
$\om$ is allowed to have poles on the unit circle is considered. Let $\Rat$ denote the space of rational complex functions, and $\Rat_0$ the subspace of strictly proper rational complex functions. We will also need the subspaces $\Rat(\mathbb{T})$ and $\Rat_0(\mathbb{T})$ of $\Rat$ consisting of the rational functions in $\Rat$ with all poles on $\BT$ and the strictly proper rational functions in $\Rat$ with all poles on $\BT$, respectively. For $\om \in \Rat$, possibly having poles on $\BT$, we define a Toeplitz-like operator $T_\omega (H^p \rightarrow H^p)$, for $1 < p<\infty$, as follows:
\begin{equation}\label{Toeplitz}
\Dom(T_\omega)\!=\! \left\{ g\in H^p \! \mid \! \omega g = f + \rho \mbox{ with } f\!\in\! L^p\!\!,\, \rho \!\in\!\textup{Rat}_0(\mathbb{T})\right\},\
T_\omega g = \mathbb{P}f.
\end{equation}
Note that in case $\om$ has no poles on $\BT$, then $\om\in L^\infty$ and the Toeplitz-like operator $T_\om$ defined above coincides with the classical Toeplitz operator $T_\om$ on $H^p$. In general, for $\om\in\Rat$, the operator $T_\om$ is a well-defined, closed, densely defined linear operator. By the Euclidean division algorithm, one easily verifies that all polynomials are contained in $\Dom(T_\om)$.  Moreover, it can be verified that  $\Dom(T_\om)$ is invariant under the forward shift operator $T_z$ and that the following classical result holds:
 $$T_{z^{-1}}T_\om T_z f = T_\om f, \qquad f\in\Dom(T_\om).$$
These basic properties are derived in Section \ref{S:Basic}.

This definition is somewhat different from earlier definitions of unbounded Toeplitz-like operators, as discussed in more detail in a separate part, later in this introduction. The fact that all polynomials are contained in $\Dom(T_\om)$, which is not the case in several of the definitions in earlier publications, enables us to determine a matrix representation with respect to the standard basis of $H^p$ and derive results on the convergence behaviour of the matrix entries; see Theorem \ref{T:Mainthm2} below.

In this paper we are specifically interested in the Fredholm properties of $T_\om$. For the case that $\om$ has no poles on $\BT$, when $T_\om$ is a classical Toeplitz operator, the operator $T_\om$ is Fredholm if and only if $\om$ has no zeroes on $\BT$, a result of R. Douglas; cf., Theorem 2.65 in \cite{BS06} and Theorem 10 in \cite{V03}. This result remains true in case $\om\in\Rat$.
We use the standard definitions of Fredholmness and Fredholm index for an unbounded operator, as
given in \cite{G66}, Section IV.2: a closed linear operator which has a finite dimensional kernel
and for which the range has a finite dimensional complement is called a Fredholm operator, and the index is defined by the difference of the dimension of the kernel and the dimension of the complement of the range.  Note that a closed Fredholm operator in a Banach space necessarily has a closed range (\cite{G66}, Corollary IV.1.13).The main results on unbounded Fredholm operators can be found in \cite{G66}, Chapters IV and V.

\begin{theorem}\label{T:Mainthm1}
Let $\omega\in\Rat$. Then $T_\omega$ is Fredholm if and only if $\omega$ has no zeroes on $\mathbb{T}$.
Moreover, in that case the index of $T_\om$ is given by
\[
\Index(T_\om)=
\sharp \left\{\begin{array}{l}\!\!\!
 \textrm{poles of } \om \textrm{ in } \overline{\BD} \textrm{ multi.}\!\!\! \\
\!\!\!\textrm{taken into account}\!\!\!
\end{array}\right\}  -
\sharp \left\{\begin{array}{l}\!\!\! \textrm{zeroes of } \om \textrm{ in } \BD  \textrm{ multi.}\!\!\! \\
\!\!\!\textrm{taken into account}\!\!\!
\end{array}\right\} .
\]
\end{theorem}

It should be noted that when we talk of poles and zeroes of $\om$ these do not include the poles or zeroes at infinity.

The result of Theorem \ref{T:Mainthm1} may also be expressed in terms of the winding number as follows: $\Index(T_\om)=-\lim_{r\downarrow 1} {\rm wind\, } (\om|r\BT)$. In the case where $\om$ is continuous on the unit circle and has no zeroes there, it is well-known that the index of the
Fredholm operator $T_\om$ is given by the negative of the winding number of the curve $\om(\BT)$ with respect to zero (see, e.g., \cite{BS06}, or  \cite{GGK02}, Theorem XVI.2.4). However, if $\om$ has poles on the unit circle, the limit
$\lim_{r\downarrow 1}$ cannot be replaced by either $\lim_{r\to 1}$ or $\lim_{r\uparrow 1}$ in this formula.

The proof of Theorem \ref{T:Mainthm1} is given in Section \ref{S:Fredholm2}. It relies heavily on the following analogue of Wiener-Hopf factorization given in Lemma \ref{L:factor}:
for $\om\in\Rat$ we can write
$\om(z)=\om_-(z)(z^\kappa \om_0(z))\om_+(z)$ where $\kappa$ is the difference between the number of zeroes of $\om$ in $\BD$ and the number of poles of $\om$ in $\BD$, $\om_-$ has no poles or zeroes outside $\BD$, $\om_+$ has no poles or zeroes inside $\overline{\BD}$ and
$\om_0$ has all its poles and zeroes on $\BT$. Based on the choice of the domain as in \eqref{Toeplitz} it can then be shown that $T_\om=T_{\om_-}T_{z^\kappa\om_0}T_{\om_+}$.
This factorization eventually allows to reduce the proof of Theorem \ref{T:Mainthm1} to the case where $\om$ has only poles on $\BT$. It also allows us to characterize invertibility of $T_\om$ and to give a formula for the inverse of $T_\om$
in case it exists.

If $\om$ has only poles on $\BT$, i.e., $\om\in\Rat(\BT)$, then we have a more complete description of $T_\om$ in case it is a Fredholm operator. Here and in the remainder of the paper, we let $\cP$ denote the space of complex polynomials in $z$, i.e., $\cP=\BC[z]$, and $\cP_n\subset\cP$ the subspace of polynomials of degree at most $n$.

\begin{theorem}\label{T:Mainthm1a}
Let $\omega\in\Rat(\BT)$, say $\om=s/q$ with $s,q\in\cP$ co-prime. Then $T_\om$ is Fredholm if and only if $\om$ has no zeroes on $\BT$. Assume $s$ has no roots on $\BT$ and factor $s$ as $s=s_- s_+$ with $s_-$ and $s_+$ having roots only inside and outside $\BT$, respectively. Then
\begin{equation}\label{DRK}
\begin{aligned}
& \Dom(T_\om) = qH^p+\cP_{\deg(q)-1},\quad \Ran(T_\om)=  s H^p+\wtil\cP,\\
&\kernel (T_\omega) = \left\{\frac{r_0}{s_+} \mid \deg(r_0) < \deg(q) - \deg(s_-) \right\}.
\end{aligned}
\end{equation}
Here $\wtil\cP$ is the subspace of $\cP_{\deg(s)-1}$ given by
\[
\wtil\cP = \{ r\in\cP \mid r q = r_1 s + r_2 \mbox{ for } r_1,r_2\in\mathcal{P}_{\deg(q)-1}\}.
\]
Moreover, a complement of $\Ran(T_\om)$ in $H^p$ is given by $\cP_{\deg(s_-)-\deg(q)-1}$ \textup{(}to be interpreted as $\{0\}$ in case $\deg(s_-)\leq \deg (q)$\textup{)}. In particular, $T_\om$ is either injective or surjective, and both injective and surjective if and only if $\deg(s_-)=\deg (q)$, and the Fredholm index of $T_\om$ is given by
$$
\Index(T_\om) =
\sharp\left\{
\begin{array}{l}\!\!\!
\textrm{poles of }\omega \textrm{ multi.}\!\!\! \\
\!\!\!\textrm{taken into account}\!\!\!
\end{array}\right\} -
\sharp\left\{
\begin{array}{l}\!\!\!
\textrm{zeroes of } \omega \textrm{ in }\mathbb{D} \textrm{ multi.}\!\!\!\\
\!\!\!\textrm{taken into account}\!\!\!
\end{array}\right\}
.
$$
\end{theorem}

The proof of Theorem \ref{T:Mainthm1a} is given in Section \ref{S:Fredholm1}. In case $\om$ has zeroes on $\BT$, so that $T_\om$ is not Fredholm, part of the claims of Theorem \ref{T:Mainthm1a} remain valid, after slight reformulation. For instance, the formula for $\kernel(T_\om)$ holds provided that the roots of $s$ on $\BT$ are included in $s_+$ (see Lemma \ref{L:kernel2}) and of the identities for $\Dom(T_\om)$ and $\Ran(T_\om)$ only one-sided inclusions are proved in case zeroes on $\BT$ are present (see Proposition \ref{P:DomRanIncl} for further detail).

Since all polynomials are in the domain of $T_\omega$ we can write down the matrix representation of $T_\omega$ with respect to the standard basis of $H^p$. It turns out that this matrix representation has the form of a Toeplitz matrix. In addition, there is an assertion on the growth of the coefficients in the upper triangular part of the matrix.

\begin{theorem}\label{T:Mainthm2}
Let $\omega\in\Rat$ possibly with poles on $\mathbb{T}$. Then we can write the matrix representation $[T_\omega]$ of $T_\omega$ with respect to the standard basis $\{z^n\}_{n=0}^\infty$ of $H^p$ as
$$\begin{array}{ll}
[T_\omega] =&
\left (
\begin{array}{cccccc}
a_0 & a_{-1} & a_{-2} & a_{-3} & a_{-4} & \cdots \\
a_1 & a_0 & a_{-1} & a_{-2} & a_{-3} &  \cdots \\
a_2 & a_1 & a_0 & a_{-1} & a_{-2} &  \cdots \\
\vdots & \multicolumn{5}{c}{\ddots} \\
\end{array}
\right )
\end{array}.
$$
In addition $a_{-j} = O(j^{M-1})$ for $j\geq 1$ where $M$ is the largest order of the poles of $\omega$ in $\mathbb{T}$ and $(a_j)_{j=0}^{\infty}\in\ell^2$.
\end{theorem}

%

In subsequent papers we will discuss further properties of the class of Toeplitz operators given by
\eqref{Toeplitz}.  In particular, in \cite{GtHJR} the spectral properties of such operators are discussed. In further subsequent papers a formula for the adjoint will be given, and several properties of the adjoint will be presented, and the matrix case will be discussed.

\paragraph{\bf Connections to earlier work on unbounded Toeplitz operators}
Several authors have considered unbounded Toeplitz operators before. In the following we shall distinguish between several definitions by using superscripts.

For $\om:\BT\rightarrow \BC$ the Toeplitz operator is defined usually by $ T_\om f = \BP \om f$
with domain given by
$
\Dom (T_\om) = \{f\in H^p \mid \om f\in L^p\}$, see e.g., \cite{HW50}. Note that for $\om$ rational with a pole on $\BT$ this is a smaller set than in our definition \eqref{Toeplitz}.
To distinguish between the two operators, we denote the classical operator by $T_\om^{\textup{cl}}$.
Hartman and Wintner have shown in \cite{HW50} that the Toeplitz operator  $T_\om^{\textup{cl}}$  is bounded if and only if its symbol is in $L^\infty$, as was established earlier by Otto Toeplitz in the case of symmetric operators.
Hartman, in \cite{H63}, investigated unbounded Toeplitz operators on $\ell^2$ (equivalently on $H^2$) with $L^2$-symbols.
The operator in \cite{H63} is given by
\[
\Dom (T_\om^{\textup{Hr}}) = \{ f\in H^2 \mid \om f = g_1 + g_2 \in L^1, g_1\in H^2, g_2\in\overline{z} \overline{H^1}\}, \quad
T_\om^{\textup{Hr}} f = g_1.
\]
Observe the similarity with the definition \eqref{Toeplitz}.
These operators are not bounded, unless $\om\in L^\infty$. Note that the class of symbols discussed in the current paper does not fall into this category, as a rational function with a pole on $\BT$ is not in $L^2$.
The Toeplitz operator $T_\om^{\textup{Hr}}$ with $L^2$-symbol is necessarily densely defined as its domain would contain the polynomials.
The operator $T_\om^{\textup{Hr}}$ is an adjoint operator and so it is closed. Necessary and sufficient conditions for invertibility have been established for the case where $\om$ is real valued on $\BT$ in terms of $\om \pm i$. Of course, $T_\om^{\textup{Hr}}$ is symmetric in this case.

In \cite{R69} Rovnyak considered a Toeplitz operator in $H^2$ with real valued $L^2$ symbol $W$ such that $\log (W) \in L^1$. The operator is symmetric and densely defined via a construction of a resolvent involving a Reproducing Kernel Hilbert Space. This leads to a self-adjoint operator and clearly, the construction is very different from the approach taken in the current paper.

Janas, in \cite{J91}, considered Toeplitz operators in the Bargmann-Segal space $B$ of Gaussian square integrable entire functions in $\BC^n$. The Bargmann-Segal space is also referred to as the Fock space or the Fisher space in the literature. The symbol of the operator is a measurable function. A Toeplitz-like operator, ${T}_\om^{\textup{J}}$, is introduced as
\[
\Dom (T_\om^{\textup{J}}) =\! \{ f\in B \mid \om f\! = h + r,\  h\in B,\! \int\!\! r\overline{p}\, \textup{d}\mu =0, \mbox{ for all } p\in\cP \}, \ T_\om^{\textup{J}} f = h.
\]
Again, observe the similarity with the definition \eqref{Toeplitz}. Consider also the operator
$T_\om^{\textup{cl,B}}$ on the domain $\{f\in B\mid \om f\in L_2(\mu)\}$ with
$T_\om^{\textup{cl,B}}f=\mathbb{P} \om f$.
It is shown in \cite{J91} that
\begin{enumerate}
\item $T_\om^{\textup{cl,B}} \subset T_\om^{\textup{J}}$, i.e. $T_\om^{\textup{J}}$ is an extension of the Toeplitz operator $T_\om^{\textup{cl,B}}$,
\item $T_\om^{\textup{J}}$ is closed,
\item $T_\om^{\textup{cl,B}}$ is closable whenever $\Dom (T_\om^{\textup{cl,B}})$ is dense in $B$,
\item if $\cP \subset \Dom (T_\om^{\textup{cl,B}})$ and $\om$ is an entire function then $T_\om^{\textup{cl,B}} = T_\om^{\textup{J}}$.
\end{enumerate}

Let $N^+$ be the Smirnov class of holomorphic functions in $\BD$ that consists of quotients of functions in $H^\infty$ with the denominator an outer function. Note that a nonzero function $\om\in N^+$ can always be written uniquely as $\om= \frac{b}{a}$ where $a$ and $b$ are in the unit ball of $H^\infty$, $a$ an outer function, $a(0) > 0$ and $\vert a\vert^2 + \vert b\vert^2 = 1$ on $\BT$, see \cite[Proposition 3.1]{S08}. This is called the canonical representation of $\om\in N^+$. For $\om\in N^+$ the Toeplitz operator $T_\om^{\textup{He}}$ on $H^2$ is defined by Helson in \cite{H90} and Sarason in \cite{S08} as the multiplication operator with domain
\[
\Dom (T_\om^{\textup{He}}) = \{ f \in H^2 \mid \om f\in H^2\}
\]
and so this is a closed operator.
Note that although a rational function with poles only on the unit circle is in the Smirnov class, the definition of the domain in \eqref{Toeplitz} is different from the one used in \cite{S08}. In fact, for $\om\in \Rat(\BT)$, the operator \eqref{Toeplitz} is an extension of
the operator $T_\om^{\textup{He}}$, i.e., $T_\om^{\textup{He}} \subset T_\om$.
In \cite{S08} it is shown that if $\Dom (T_\om^{\textup{He}})$ is dense in $H^2$ then $\om\in N^+$. Also, if $\om$ has canonical representation $\om = \frac{b}{a}$ then $\Dom (T_\om^{\textup{He}}) = aH^2$; compare with \eqref{DRK} to see the difference. By extending our domain as in \eqref{Toeplitz}, our Toeplitz-like operator $T_\om$ is densely defined for any $\om\in\Rat$, i.e., poles inside $\BD$ are allowed.



Helson in \cite{H90} studied $T_\om^{\textup{He}}$ in $H^2$ where $\om\in N^+$ with $\om$ real valued on $\BT$. In this case $T_\om^{\textup{He}}$ is symmetric, and Helson showed among other things that $T_\om^{\textup{He}}$ has finite deficiency indices if and only if $\om$ is a rational function.

\paragraph{\bf Overview}
The paper consists of six sections, including the current introduction. In Section \ref{S:Basic} we prove several basic results concerning the Toeplitz-like operator $T_\om$. In the following section, Section \ref{S:Division}, we look at division with remainder by a polynomial in $H^p$. The results in this section form the basis of many of the proofs in subsequent sections, and may be of independent interest. Section \ref{S:Fredholm1} is devoted to the case where $\om$ is in $\Rat(\BT)$. Here we prove Theorem \ref{T:Mainthm1a}. In Section \ref{S:Fredholm2} we prove the Fredholm result for general $\om\in\Rat$, Theorem \ref{T:Mainthm1}, and in Section \ref{S:matrix} we prove Theorem \ref{T:Mainthm2} on the matrix representation of $T_\om$. Finally, in Section \ref{S:Examples} we discuss three examples that illustrate the main results of the paper.

\paragraph{\bf Notation}
We shall use the following notation, most of which is standard:
$\cP$  is the space of polynomials (of any degree) in one variable; $\cP_n$ is the subspace of polynomials of degree at most $n$.
Throughout, $K^p$ denotes the standard complement of $H^p$ in $L^p$;
$\cW_+$ denotes the analytic Wiener algebra on $\BD$, that is, power series $f(z)=\sum_{n=0}^\infty f_n z^n$ with absolutely summable Taylor coefficients, hence analytic on $\BD$ and continuous on $\overline{\BD}$. In particular, $\cP\subset \cW_+\subset L^p$ for each $p$.

\section{Basic properties of $T_\omega$}\label{S:Basic}

In this section we derive some basic properties of the Toeplitz-like operator $T_\om$ as defined in \eqref{Toeplitz}. The main result is the following proposition.

\begin{proposition}\label{P:welldefclosed}
Let $\om\in\Rat$, possibly having poles on $\BT$. Then $T_\om$ is a well-defined closed linear operator on $H^p$ with a dense domain which is invariant under the forward shift operator $T_z$. More specifically, the subspace $\cP$ of polynomials is contained in $\Dom(T_\om)$.
Moreover, $T_{z^{-1}}T_\om T_z f= T_\om f$ for $f\in\Dom(T_\om)$.
\end{proposition}

The proof of the well-definedness relies on the following well-known result.

\begin{lemma}\label{L:h2_rat}
Let $\psi\in\Rat$ have a pole on $\BT$. Then $\psi\not\in L^p$. In particular, the intersection of $\Rat_0(\BT)$ and $L^p$ consists of the zero function only.
\end{lemma}

Indeed, if $\psi\in\Rat$ has a pole at $\al\in\BT$ of order $n$, then $|\psi(z)|\sim |z-\al|^{-n}$
as $z\to\alpha$, and therefore the integral $\int_\BT |\psi(z)|^p\,\textup{d}z $ diverges.

\begin{proof}[\bf Proof of well-definedness claim of Proposition \ref{P:welldefclosed}]
Let $g\in\Dom(T_\om)$ and assume $f_1, f_2\in L^p$ and $\rho_1,\rho_2\in \Rat_0(\BT)$ such that $f_1+\rho_1=\om g=f_2+\rho_2$. Then $f_1 - f_2 = \rho_2 - \rho_1 \in L^p \cap \Rat_0(\BT)$. By Lemma \ref{L:h2_rat} we have $f_1 - f_2 = \rho_2 - \rho_1=0$, i.e., $f_1=f_2$ and $\rho_1=\rho_2$. Hence $f$ and $\rho$ in the definition of $\Dom(T_\om)$ are uniquely determined. From this and the definition of $T_\om$ it is clear that $T_\om$ is a well-defined linear operator.
\end{proof}

In order to show that $T_\om$ is a closed operator, we need the following  alternative formula for $\Dom(T_\om)$ for the case where $\om \in\Rat(\BT)$.

\begin{lemma}\label{L:domain}
Let $\om\in\Rat(\BT)$, say $\om=s/q$ with $s,q\in\cP$ co-prime. Then
\begin{equation}\label{Rat0Tdom}
\Dom(T_\om)=\{ g\in H^p \colon \om g =h + r/q,\, h\in H^p,\, r\in\cP_{\deg(q)-1}\},\ T_\om g =h.
\end{equation}
Moreover, $\Dom(T_\om)$ is invariant under the forward shift operator $T_z$ and
 $$T_{z^{-1}}T_\om T_z f = T_\om f\qquad f\in\Dom(T_\om).$$
\end{lemma}

\begin{proof}[\bf Proof]
Assume $g\in H^p$ with $\omega g =h+ r/q$, where  $h\in H^p$ and $r\in\cP_{\deg(q)-1}$. Since $H^p\subset L^p$ and $r/q\in \Rat_0(\mathbb{T})$, clearly $g \in \Dom(T_\omega)$ and $T_\omega g=\mathbb{P}h=h$.  Thus it remains to prove the reverse implication.

Assume  $g \in \Dom(T_\omega)$, say  $\omega g =f+ \rho$, where $f\in L^p$ and $\rho\in \Rat_0(\BT)$. Since $\rho\in \Rat_0(\BT)$, we can write $q \rho$ as $q \rho=r_0+\rho_0$ with $r_0\in\cP_{\deg(q)-1}$ and $\rho_0\in\Rat_0(\BT)$. Then
\[
s g = q \om g= qf+q\rho= qf +r_0 +\rho_0, \quad \mbox{i.e.,}\quad
\rho_0=s g-qf-r_0\in L^p.
\]
By Lemma \ref{L:h2_rat} we find that $\rho_0\equiv0$. Thus $s g=qf +r_0$. Next write $f=h+k$ with $h\in H^p$ and $k\in K^p$. Then $qk$ has the form $qk =r_1+k_1$ with $r_1\in\cP_{\deg(q)-1}$ and $k_1\in K^p$. Thus
\[
sg=qh +qk+r_0= qh + k_1 + r_1 +r_0, \quad \mbox{i.e.,}\quad
k_1= sg-qh-r_1-r_0\in H^p.
\]
Since also $k_1\in K^p$, this shows that $k_1\equiv 0$, and we find that $sg=qh+ r$ with $r=r_0+r_1\in\cP_{\deg(q)-1}$. Dividing by $q$ gives $\om g=h+r/q$ with $h\in H^p$ as claimed.

Finally, we prove that $\Dom(T_\om)$ is invariant under $T_z$. Let $f\in \Dom(T_\om)$, say $s f= qh+r$ with $h\in H^p$ and $r\in\cP_{\deg(q)-1}$. Then $s z f= q zh+zr$. Now write $zr= c q+ r_0$ with $c\in\BC$ and $r_0\in \cP_{\deg(q)-1}$. Then $s z f= q (zh +c)+r_0$ is in $q H^p + \cP_{\deg(q)-1}$. Thus $zf\in\Dom(T_\om)$, and $T_\om T_z f = zh+c$. Hence $T_{z^{-1}}T_\om T_z f= h=T_\om f$ as claimed.
\end{proof}

\begin{lemma}\label{L:reducRatT}
Let $\om\in\Rat$. Then $\om=\om_0+\om_1$ with $\om_0\in\Rat_0(\BT)$ and $\om_1\in\Rat$ with no poles on $\BT$. Moreover, $\om_0$ and $\om_1$ are uniquely determined by $\om$ and the poles of $\om_0$ and $\om_1$ correspond to the poles of $\om$ on and off $\BT$, respectively.
\end{lemma}

\begin{proof}[\bf Proof]
The existence of the decomposition follows from the partial fraction decomposition of $\om$
into the sum of a polynomial and elementary fractions of the form $c/(z-z_k)^n$.

To obtain the uniqueness, split $\omega_1$ into the sum of a strictly proper rational function
$\nu_1$ and a polynomial $p_1$. Assume also $\om=\om_0^\prime + \nu_1^\prime +p_1^\prime$ with $\om_0^\prime$ in $\Rat_0(\BT)$, $\nu_1^\prime\in\Rat_0$ with no poles on $\BT$ and
$p_1^\prime$ a polynomial. Then $(\om_0-\om_0^\prime)+(\om_1-\om_1^\prime)=p_1^\prime -p_1$ is in $\Rat_0\cap \cP$, and hence is zero. So $p_1=p_1^\prime$. Then $\om_0-\om_0^\prime=\om_1^\prime-\om_1$ is in $\Rat_0$ and has no poles on $\BC$, and hence it is the zero function.
\end{proof}

\begin{proof}[\bf Proof of closedness claim of Proposition \ref{P:welldefclosed}]
By Lemma \ref{L:reducRatT}, $\om\in\Rat$ can be written as $\om=\om_0+\om_1$ with $\om_0\in\Rat_0(\BT)$ and $\om_1\in\Rat$ with no poles on $\BT$, hence $\om_1\in L^\infty$. Then $T_\om=T_{\om_0} + T_{\om_1}$ and $T_{\om_1}$ is bounded on $H^p$. It follows that $T_\om$ is closed if and only if $T_{\om_0}$ is closed. Hence, without loss of generality we may assume $\om\in\Rat_0(\BT)$, which we will do in the remainder of the proof.

Say $\omega=s/q$ with $s,q\in\mathcal{P}$ co-prime, $q$ having  roots  only on  $\mathbb{T}$ and $\deg(s)<\deg(q)$. Let $g_1, g_2, \ldots $ be a sequence in $\Dom(T_\omega)$ such that in $H^p$ we have
\begin{equation}\label{conv}
g_n\to g\in H^p \ands T_\omega g_n\to h\in H^p \quad\mbox{as $n\to\infty$}.
\end{equation}
We have to prove that $g \in \Dom(T_\omega)$ and $T_\omega g =h$. Applying Lemma \ref{L:domain} above, we know that $\omega g_n=h_n+r_n/q$ with $h_n\in H^p$ and $r_n\in\cP_{\deg(q)-1}$. Moreover $h_n=T_\om g_n\to h$. Using \eqref{conv} it follows that
\[
r_n =sg_n-qh_n\to sg -q h=:r  \quad \mbox{as $n\to \infty$, with convergence in $H^p$}.
\]
Since $\deg(r_n)<\deg(q)$ for each $n$,  it follows that  $r=\lim_{n\to \infty} r_n$ is also a polynomial with $\deg(r)<\deg(q)$. Thus $r/q\in \Rat_0\mathbb(\BT)$, and $r=sg -qh$ implies that $\omega g= h+r/q$. Thus $g \in \Dom(T_\omega)$ and $T_\omega g =h$. We conclude that $T_\omega$ is closed.
\end{proof}

\begin{proof}[\bf Proof of Proposition \ref{P:welldefclosed}]
In the preceding two parts of the proof we showed all claims except that $\Dom(T_\om)$ contains $\cP$ and is invariant under $T_z$. Again write $\om$ as $\om=\om_0+\om_1$ with $\om_0\in\Rat_0(\BT)$ and $\om_1\in\Rat$ with no poles on $\BT$. Let $r\in\cP$. Then $\om r =\om_0 r + \om_1 r$. We have $\om_1 r\in \Rat$ with no poles on $\BT$, hence $\om_1 r\in L^p$. By Euclidean division, $\om_0 r=\psi +r_0$ with $\psi\in \Rat_0(\BT)$ (having the same denominator as $\om_0$) and $r_0\in\cP\subset L^p$. Hence $\om r \in L^p + \Rat_0(\BT)$, so that $r\in \Dom(T_\om)$. This shows $\cP\subset \Dom(T_\om)$. Finally, we have $\Dom(T_\om)= \Dom(T_{\om_0})$ and it follows by the last claim of Lemma \ref{L:domain} that $\Dom(T_{\om_0})$ is invariant under $T_z$.
\end{proof}

\section{Intermezzo:\ Division with remainder by a polynomial in $H^p$}\label{S:Division}

Let $s\in\cP$, $s\not\equiv 0$. The Euclidean division algorithm says that for any $v\in\cP$ there exist unique $u,r\in\cP$ with $v=us+r$ and $\deg(r)<\deg(s)$. If $\deg(v)\geq \deg(s)$, then $\deg(v)=\deg(s)+\deg(u)$. We can reformulate this as:
\[
\cP=s \cP\dot{+} \cP_{\deg(s)-1} \ands
\cP_n=s\cP_{n-\deg(s)} \dot{+}\cP_{\deg(s)-1},\ \ n\geq \deg(s),
\]
with $\dot{+}$ indicating direct sum. What happens when $\cP$ is replaced with a class of analytic functions, say by $H^p$, $p\geq 0$? That is, for $s\in\cP$, $s\not\equiv 0$, when do we have
\begin{equation}\label{Hp-euclid}
H^p= s H^p+\cP_{\deg(s)-1}?
\end{equation}
Since $\cP\subset H^p$, we know that $\cP=s \cP\dot{+} \cP_{\deg(s)-1}\subset s H^p+\cP_{\deg(s)-1}$. Hence $s H^p+\cP_{\deg(s)-1}$ contains a dense (non-closed) subspace of $H^p$. Thus question \eqref{Hp-euclid} is equivalent to asking whether $s H^p+\cP_{\deg(s)-1}$ is closed. The following theorem provides a full answer to the above question.

\begin{theorem}\label{T:main1}
Let $s\in\cP$, $s\not\equiv 0$. Then $H^p= s H^p+\cP_{\deg(s)-1}$ if and only if $s$ has no roots on the unit circle $\BT$.
\end{theorem}

Another question is, even if $s$ has no roots on $\BT$, whether $s H^p+\cP_{\deg(s)-1}$ is a direct sum. This does not have to be the case. In fact, if $s$ has only roots outside $\BT$, then $1/s\in H^\infty$ and $s H^p=H^p$, so that $s H^p+\cP_{\deg(s)-1}$ is not a direct sum, unless if $s$ is constant. Clearly, a similar phenomenon occurs if only part of the roots of $s$ are outside $\BT$. In case all roots of $s$ are inside $\BT$, then the sum is a direct sum.

\begin{proposition}\label{P:directsum}
Let $s\in\cP$, $s\not\equiv 0$ and having no roots on $\BT$. Write $s=s_- s_+$ with $s_-,s_+\in\cP$ having roots inside and outside $\BT$, respectively. Then $H^p=s H^p+\cP_{\deg(s_-)-1}$ is a direct sum decomposition of $H^p$. In particular, $s H^p+\cP_{\deg(s)-1}$ is a direct sum if and only if $s$ has all its roots inside $\BT$.
\end{proposition}


We also consider the question whether there are functions in $H^p$ that are not in $s H^p+\cP_{\deg(s)-1}$ and that can be divided by another polynomial $q$. This turns out to be the case precisely when $s$ has a root on $\BT$ which is not a root of $q$.

\begin{theorem}\label{T:main2}
Let $s,q\in\cP$, $s,q\not\equiv 0$. Then there exists a $f\in q H^p$ which is not in $s H^p+\cP_{\deg(s)-1}$ if and only if $s$ has a root on $\BT$ which is not a root of $q$.
\end{theorem}

In order to prove the above results we first prove a few lemmas.

\begin{lemma}\label{L:1}
Let $s\in \cP$ and $\al\in\BC$ a root of $s$. Then $s H^p+\cP_{\deg(s)-1}\subset (z-\al)H^p +\BC$.
\end{lemma}

\begin{proof}[\bf Proof]
Since $s(\al)=0$, we have $s(z)=(z-\al)s_0(z)$ for some $s_0\in\cP$, $\deg(s_0)=\deg(s)-1$. Let $f=s g+r \in s H^p+\cP_{\deg(s)-1}$. Then $r(z)=(z-\al)r_0(z)+c$ for a $r_0\in\cP$ and $c\in\BC$. This yields
\[
f(z)=s(z) g(z)+r(z)=(z-\al)(s_0(z) g(z)+r_0(z))+c\in (z-\al) H^p +\BC. \qedhere
\]
\end{proof}

\begin{lemma}\label{L:3}
Let $\al\in\BT$. Then there exists a $f\in\cW_+$ such that $f\not\in (z-\al) H^p+\BC$.
\end{lemma}

\begin{proof}[\bf Proof]
By rotational symmetry we may assume without loss of generality that $\alpha=1$.
Let $h_n\downarrow 0$ such that $h(z)=\sum_{n=0}^\infty h_n z^n$ is analytic on $\BD$ but $h\not\in H^p$. Define $f_0,f_1,\ldots$ recursively by
\[
f_0=-h_0,\quad f_{n+1}=(h_n-h_{n+1}),\ n\geq 0.
\]
Then $f(z)=(z-1)h(z)$ and $\sum_{k=0}^N |f_k|=2h_0-h_N\to 2h_0$. Hence the Taylor coefficients of $f(z)=\sum_{k=0}^\infty f_k z^k$ are absolutely summable and thus $f\in\cW_+$.

Now assume $f\in (z-1)H^p+\BC$, say $f=(z-1)g +c$ for $g\in H^p$ and $c\in\BC$. Then
$h=g+c/(z-1)$. Since the Taylor coefficients of $c/(z-1)$ have to go to zero, we obtain $c=0$ and
$h=g$, which contradicts the assumption $h\notin H^p$.
%
\end{proof}

\begin{proof}[\bf Proof of Theorem \ref{T:main1}]
Assume $s$ has no roots on $\BT$. Since $s\in \cP\subset H^\infty$, we know from Theorem 8 of \cite{V03} that the range of the multiplication operator of $s$ on $H^p$ is closed (i.e., $s H^p$ closed in $H^p$) if and only if $|s|$ is bounded away from zero on $\BT$. Since $s$ is a polynomial, the latter is equivalent to $s$ having no roots on $\BT$. Hence $s H^p$ is closed. Since $\cP_{\deg(s)-1}$ is a finite dimensional subspace of $H^p$, and thus closed, we obtain that $s H^p+\cP_{\deg(s)-1}$ is closed \cite[Chapter 3, Proposition 4.3]{C90}. Also, $s H^p+\cP_{\deg(s)-1}$ contains the dense subspace $\cP$ of $H^p$, therefore $s H^p+\cP_{\deg(s)-1}= H^p$.

Conversely, assume $s$ has a root $\al\in\BT$. Then by Lemmas \ref{L:1} and \ref{L:3} we know  $s H^p+\cP_{\deg(s)-1}\subset (z-\al)H^p+\BC\neq H^p$.
\end{proof}

\begin{proof}[\bf Proof of Proposition \ref{P:directsum}]
Assume $s\in\cP$ has no roots on $\BT$. Write $s=s_- s_+$ with $s_-,s_+\in\cP$, $s_-$ having only roots inside $\BT$ and $s_+$ having only roots outside $\BT$. Assume $s$ has roots outside $\BT$, i.e., $\deg(s_+)>0$. Then $1/s_+$ is in $H^\infty$ and $s_+ H^p=H^p$ and hence $s H^p=s_- H^p$. Using Theorem \ref{T:main1}, this implies that
\[
H^p= s_- H^p+\cP_{\deg(s_-)-1}= s H^p+\cP_{\deg(s_-)-1}.
\]
Next we show that $s H^p+\cP_{\deg(s_-)-1}$ is a direct sum. Let $f=s h_1+r_1 =s h_2+r_2\in s H^p +\cP_{\deg(s_-)-1}$ with $h_1,h_2\in H^p$, $r_1,r_2\in\cP_{\deg(s_-)-1}$. Then $r_1-r_2= s(h_2-h_1)$. Clearly, each root $\al$ of $s_-$ with multiplicity $n$, is also a root of $s$ with multiplicity $n$. Evaluate both sides of $r_1-r_2= s(h_2-h_1)$ at $\al$, possible since $\al\in\BD$, as well as the identities obtained by taking derivatives on both sides up to order $n-1$, this yields $\frac{\textup{d}^m}{\textup{d}z^m} (r_1-r_2)(\al)=0$ for $m=0\ldots n-1$. Since $\deg(r_1-r_2)<\deg(s_-)$, this can only occur when $r_1-r_2\equiv 0$, i.e., $r_1=r_2$. We thus arrive at $s(h_2-h_1)\equiv 0$. Since $s$ has no roots on $\BT$, we have $1/s\in L^\infty$ so that $h_2-h_1= s^{-1} s(h_2-h_1)\equiv 0$ as a function in $L^p$. Hence $h_1=h_2$ in $L^p$, but then also $h_1=h_2$ in $H^p$. Hence we have shown $s H^p+\cP_{\deg(s_-)-1}$ is a direct sum.

In case $s$ has all its roots inside $\BT$, we have $s=s_-$ and thus $\cP_{\deg(s)-1}=\cP_{\deg(s_-)-1}$ so that $sH^p +\cP_{\deg(s)-1}$ is a direct sum. Conversely, if $s$ has a root outside $\BT$, we have $\deg(s_-)<\deg(s)$ and the identity $sH^p +\cP_{\deg(s)-1}=sH^p +\cP_{\deg(s_-)-1}$ shows that any $r\in\cP_{\deg(s)-1}$ with $\deg(r)\geq\deg(s_-)$ can be written as $r=0+r\in sH^p +\cP_{\deg(s)-1}$ and as $r=sh+r'\in sH^p +\cP_{\deg(s)-1}$ with $\deg(r')<\deg(s_-)$ and $h\in H^p$, $h\not \equiv 0$. Hence $sH^p +\cP_{\deg(s)-1}$ is not a direct sum.
\end{proof}

\begin{proof}[\bf Proof of Theorem \ref{T:main2}]
Assume all roots of $s$ on $\BT$ are also roots of $q$. Let $f=q \wtil{f} \in q H^p$. Factor $s=s_+s_0s_-$ as before. Then $q=s_0 \hat{q}$ for some $\hat{q}\in\cP$. From Theorem \ref{T:main1} we know that $s_-s_+ H^p +\cP_{\deg(s_-s_+)-1}=H^p$. Hence $\hat{q}\wtil{f}=s_-s_+ \what{f}+r$ with $\what{f}\in H^p$ and $r\in\cP$ with $\deg(r)< \deg(s_-s_+)$. Thus
\[
fq\wtil{f}=s_0 \hat{q}\wtil{f} =s \what{f} + s_0 r \in s H^p+\cP_{\deg(s)-1},
\]
where we used $\deg(s_0 r)=\deg(s_0)+\deg(r)<\deg(s_0)+\deg(s_-s_+)=\deg(s)$. Hence $q H^p\subset sH^p+\cP_{\deg(s)-1}$.

Conversely, assume $\al\in\BT$ such that $s(\al)=0$ and $q(\al)\neq 0$. By Lemma \ref{L:3} there exists a $\wtil{f}\in \cW_+\subset H^p$ which is not in $(z-\al) H^p+\BC$, and hence not in $s H^p+\cP_{\deg(s)-1}$, by Lemma \ref{L:1}. Now set $f=q\wtil{f}\in qH^p$. We have $q(z)=(z-\al)q_1(z) +c_1$ for a $q_1\in\cP$ and $c_1=q(\al)\neq 0$. Assume $f\in (z-\al) H^p+\BC$, say $f(z)= (z-\al)g(z)+c$ for a $g\in H^p$ and $c\in\BC$. Then
\[
((z-\al)q_1(z) +c_1) \wtil{f}(z)=q(z)\wtil{f}(z)=f(z)=(z-\al)g(z)+c.
\]
Hence $\wtil{f}(z)= (z-\al)(g(z)-q_1(z)\wtil{f}(z))/c_1 + c/c_1$, $z\in\BD$, which shows $\wtil{f}\in (z-\al)H^p+\BC$, in contradiction with our assumption. Hence $f\not\in (z-\al)H^p+\BC$. This implies, once more by Lemma \ref{L:1}, that there exists a $f\in qH^p$ which is not in $s H^p+\cP_{\deg(s)-1}$.
\end{proof}

The following lemma will be useful in the sequel.

\begin{lemma}\label{L:invar}
Let $q,s_+\in\cP$, $q,s_+\not\equiv0$ be co-prime with $s_+$ having roots only outside $\BT$. Then $s_+^{-1} (qH^p +\cP_{\deg(q)-1})=qH^p +\cP_{\deg(q)-1}$.
\end{lemma}

\begin{proof}[\bf Proof]
Set $\cR:=s_+^{-1} (qH^p +\cP_{\deg(q)-1})$.
Since $s_+$ has only roots outside $\BT$, we have $s_+^{-1}\in H^\infty$ and $s_+^{-1} H^p=H^p$. Thus
\[
\cR=s_+^{-1}(qH^p +\cP_{\deg(q)-1})=qH^p +s_+^{-1}\cP_{\deg(q)-1}.
\]
This implies $q H^p\subset\cR$. Next we show $\cP_{\deg(q)-1}\subset\cR$. Let $r\in \cP_{\deg(q)-1}$. Since $\cP\subset qH^p +\cP_{\deg(q)-1}$, we have $r s_+\in qH^p +\cP_{\deg(q)-1}$ and thus $r= s_+^{-1}(r s_+)\in\cR$. Hence $qH^p +\cP_{\deg(q)-1}\subset \cR$.

It remains to prove  $\cR \subset qH^p +\cP_{\deg(q)-1}$. Let $g=s_+^{-1}(qh+r)\in\cR$ with $h\in H^p$ and $r\in\cP_{\deg(q)-1}$. Since $q$ and $s_+$ have no common roots, there exist polynomials $a,b\in\cP$ with $qa+s_+ b\equiv 1$ and $\deg(a)<\deg(s_+)$, $\deg(b)<\deg(q)$. Since $rb\in\cP\subset q H^p+\cP_{\deg(q)-1}$, we have
\[
s_+^{-1}r= s_+^{-1}r(qa+s_+b)=q s_+^{-1}ra +rb\in q H^p+\cP_{\deg(q)-1}.
\]
Also $q s_+^{-1}h\in q H^p$, so we have $g\in qH^p+\cP_{\deg(q)-1}$. This shows that $\cR \subset qH^p +\cP_{\deg(q)-1}$ and completes the proof.
\end{proof}

\begin{remark}
For what other Banach spaces $X$ of analytic functions on $\BD$ do the above results hold? Note that the following properties of $X=H^p$ are used:
\begin{enumerate}
\item[(1)] $\cP\subset\cW_+\subset X$, and $\cP$ is dense in $X$;

\item[(2)] $\cW_+ X \subset X$;

\item[(3)] If $g=\sum_{n=0}^\infty g_nz^n \in X$ then $g_n \rightarrow 0$;

\item[(4)] If $g\in X$ and $\alpha\in\mathbb{T}$, then $g(z/\alpha)\in X$ as well;

\item[(5)] If $s\in\cP$ has no roots on $\mathbb{T}$, then $sX$ is closed in $X$.

\end{enumerate}
To see item 3 for $X=H^p$: note that by H\"older's inequality $H^p\subset H^1$, and for $p=1$ this follows from the Riemann-Lebesgue Lemma (\cite[Theorem I.2.8]{K68}), actually a sharper statement can be made in that case by a theorem of Hardy, see
\cite[Theorem III 3.16]{K68}, which states that if $f\in H^1$ then $\sum \vert f_n\vert n^{-1} <\infty$.

Other than $X=H^p$, $1< p< \infty$, the spaces of analytic functions $A^p$ on $\BD$ with Taylor coefficients $p$-summable, c.f., \cite{L14} and reference ([1-5]) given there, also have these properties. For a function $f\in A^p$ the norm $\|f\|_{A^p}$ is defined as the $l^p$-norm of the
sequence $(\widehat{f})_k$ of Taylor coefficients of $f$. Properties (1), (3) and (4) above are straightforward, property (2) is the fact that a function in the Wiener algebra is an $l^p$ multiplier
(see e.g., \cite{L14}). It remains to prove property (5).

Let $s$ be a polynomial with no roots on $\BT$, and let $(f_n)$ be a sequence of functions in $A^p$ such that the sequence $(sf_n)$ converges to $g$ in $A^p$. We have to show the existence of an $f\in A^p$ such that
$g=sf$. Note that $f_n$ and $g$ are analytic functions, and convergence of $(sf_n)$ to $g$ in $A^p$
means that $\|\widehat{sf_n} -\widehat{g}\|_{l^p} \to 0$. Consider the Toeplitz operator $T_s:l^p\to l^p$. Then $\widehat{sf_n}=T_s\widehat{f_n}$. So $\|T_s\widehat{f_n} -\widehat{g}\|_{l^p} \to 0$.
Since $s$ has no roots on $\BT$ the Toeplitz operator $T_s$ is Fredholm and has closed range, and since $s$ is a polynomial $T_s$ is injective. Thus there is a unique $\widehat{f}\in l^p$ such that $T_s\widehat{f}=\widehat{g}$.  Now define (at least formally)
the function $f(z)=\sum_{n=0}^\infty (\widehat{f})_kz^k$. Then $\widehat{sf}=\widehat{g}$, so
at least formally $s(z)f(z)=g(z)$. It remains to show that $f$ is analytic on $\BD$. To see this,
consider $z=r$ with $0<r<1$. Then by H\"older's inequality
$$
\sum_{k=0}^\infty |(\widehat{f})_k| r^k \leq \|\widehat{f}\|_{l^p} \left(\sum_{k=0}^\infty
r^{kq}\right)^{1/q}= \|\widehat{f}\|_{l^p}\left( \frac{1}{1-r^q}\right)^{1/q},
$$
showing that the series $f(z)=\sum_{n=0}^\infty (\widehat{f})_kz^k$ is absolutely convergent
on $\BD$. Since $f(z)=\frac{g(z)}{s(z)}$ is the quotient of an analytic function and a polynomial
it can only have finitely many poles on $\BD$, and since the series for $f(z)$ converges for every
$z\in \BD$ it follows that $f$ is analytic in $\BD$.
\end{remark}

\section{Fredholm properties of $T_\omega$ for $\om\in \Rat(\BT)$}\label{S:Fredholm1}

In this section we prove Theorem \ref{T:Mainthm1a}. We start with the formula for $\kernel(T_\om)$.

\begin{lemma}\label{L:kernel2}
Let $\omega \in\Rat(\BT)$, say $\omega =s/q$ with $s,q\in\cP$ co-prime. Write $s = s_-s_0s_+$ with the roots of $s_-$, $s_0$, $s_+$  inside, on, or outside $\mathbb{T}$, respectively. Then
\begin{equation}\label{KerForm}
\kernel (T_\omega) = \left\{\frac{\hat{r}}{s_+} \mid \deg(\hat{r}) < \deg(q) - (\deg(s_-) + \deg(s_0)) \right\}.
\end{equation}
\end{lemma}

\begin{proof}[\bf Proof]
If $g =\hat{r}/s_+$ where $\deg(\hat{r}) < \deg(q) - (\deg(s_-) + \deg(s_0))$, then $sg = s_-s_0\hat{r}$ which is a polynomial with $\deg(s_-s_0\hat{r}) < \deg(q)$. Thus $\om g= s_-s_0\hat{r}/q\in\Rat_0(\BT)$ which implies that $g\in\Dom(T_\om)$ and $T_\om g =0$. Hence $g\in \ker (T_\omega)$. This proves the inclusion $\supset$ in the identity \eqref{KerForm}.

Conversely suppose $g\in \ker T_\omega$. Then $T_\omega g = 0$, i.e., by Lemma \ref{L:domain} we have $\omega g = \hat{r}/q$ or equivalently $sg = \hat{r}$ for some $\hat{r}\in\cP_{\deg(q)-1}$. Hence $s_-s_0(s_+g) =sg= \hat{r}$. Thus $g=\wtil{r}/s_+$ with $\wtil{r}:=s_+g\in H^p$. Note that $s_-s_0 \widetilde{r} = \hat{r}$, so that $\widetilde{r}= \hat{r}/(s_-s_0)$. Since $\wtil{r}\in H^p$ and $s_-s_0$ only has roots in $\overline{\BD}$, the identity $\widetilde{r}= \hat{r}/(s_-s_0)$ can only hold in case $s_-s_0$ divides $\hat{r}$, i.e., $\hat{r}=s_-s_0 r_1$ for some $r_1\in\cP$. Then $\widetilde{r}=r_1\in\cP$ and we have
\[
\deg(\wtil{r})=\deg(s_+g) = \deg(\hat{r}) - \deg(s_-s_0) < \deg(q) - (\deg (s_-)+\deg(s_0)).
\]
Hence $g$ is included in the right hand side of \eqref{KerForm}, and we have also proved the inclusion $\subset$. Thus \eqref{KerForm} holds.
\end{proof}

We immediately obtain the following corollaries.

\begin{corollary}\label{C:kerdim}
Let $\omega\in\Rat(\BT)$. Then
\begin{align*}
&\dim (\kernel (T_\omega)) = \\
& \qquad  = \max \left\{0, \sharp\left\{
\begin{array}{l}\!\!\!
\textrm{poles of }\omega, \textrm{ multi.}\!\!\!\! \\
\!\!\!\textrm{taken into account}\!\!\!\!
\end{array}\right\}
- \sharp\left\{
\begin{array}{l}\!\!\!
\textrm{zeroes of } \omega \textrm{ in }\overline{\mathbb{D}},\!\textrm{ multi.}\!\!\!\!\\\
\!\!\!\textrm{ taken into account}\!\!\!\!
\end{array}\right\}
\right \}.
\end{align*}
In particular, $T_\om$ is injective if and only if the number of zeroes of $\om$ inside $\overline\BD$ is greater than or equal to the number of poles of $\om$ (all on $\BT$), in both cases with multiplicity taken into account.
\end{corollary}

\begin{corollary}
Let $\omega\in\Rat(\BT)$ with all zeroes inside $\overline{\BD}$. Then
\[
\kernel (T_\om)=\{r\mid \deg(r)<\deg(q)-\deg(s)\}\subset\cP.
\]
\end{corollary}

\begin{corollary}\label{C:Rat0Inject}
Let $\omega\in\Rat_0(\BT)$, say $\om=s/q$ with $s,q\in\cP$ co-prime. Then  $\cP_{\deg(q)-\deg(s)-1}\subset \kernel (T_\om)$ and thus $T_\om$ is not injective.
\end{corollary}

Next we prove the inclusions for $\Dom(T_\om)$ and $\Ran(T_{\om})$ in \eqref{DRK}.

\begin{proposition}\label{P:DomRanIncl}
Let $\omega\in \Rat(\mathbb{T})$, say $\om=s/q$ with $s,q\in\cP$ co-prime. Then
\begin{equation}\label{DomRanIncl}
\begin{aligned}
qH^p+\cP_{\deg(q)-1}&\subset \Dom(T_\om);\\
T_\om (qH^p+\cP_{\deg(q)-1})&=s H^p+\wtil\cP\subset \Ran(T_\om),
\end{aligned}
\end{equation}
where $\wtil\cP$ is the subspace of $\cP$ given by
\begin{equation}\label{tilP}
\wtil\cP = \{ r\in\cP \mid r q = r_1 s + r_2 \mbox{ for } r_1,r_2\in\mathcal{P}_{\deg(q)-1}\}\subset \cP_{\deg(s)-1}.
\end{equation}
\end{proposition}

\begin{proof}[\bf Proof]
We start with the first inclusion of \eqref{DomRanIncl}. Let $g \in qH^p + \mathcal{P}_{\deg(q)-1}$, i.e., $g = qh + r_1$ where $h\in H^p$ and $\deg(r_1) < \deg(q)$. Write $sr_1=rq+r_2$ with $\deg(r_2)<\deg(q)$. Then
\begin{align*}
 \deg(r) +\deg(q)&=\deg(rq)=\deg (rq+r_2)=\deg(sr_1)\\
 &=\deg(s)+\deg(r_1)<\deg(s)+\deg(q).
\end{align*}
Hence $\deg(r)<\deg (s)$, and we have
\[
\om g= sh +\frac{sr_1}{q} = (sh+ r) +\frac{r_2}{q}\in H^p + \Rat_0(\BT).
\]
Hence $g\in \Dom(T_\om)$ and $T_\om g = sh+ r\subset s H^p +\cP_{\deg(s)-1}$. This proves the first inclusion in \eqref{DomRanIncl}.

Further, observe that $sr_1=rq+r_2$ implies $rq=sr_1-r_2$ and we have $\deg(r_1)<\deg(q)$ and $\deg(r_2)<\deg(q)$, so that $r\in\wtil{\cP}$. This gives the inclusion $T_\om (qH^p+\cP_{\deg(q)-1})\subset s H^p+\wtil\cP$.

To complete the proof of \eqref{DomRanIncl} it remains to prove the reverse inclusion. Let $f\in s H^p+\wtil\cP$, say $f=s h+r$ with $h\in H^p$, $r\in\wtil\cP$. Hence $qr=r_1s + r_2$ with $r_1,r_2\in\cP_{\deg(q)-1}$. We seek $g\in q H^p+\cP_{\deg(q)-1}$ and $\wtil{r}\in\cP_{\deg(q)-1}$ such that $\om g=f+\wtil{r}/q$, or equivalently
\[
sg=qf+\wtil{r}= qsh+qr+\wtil{r}=sqh+sr_1+r_2 +\wtil{r}=s(qh+r_1)+ r_2 +\wtil{r}.
\]
Since $\deg(r_2)<\deg(q)$, this is clearly satisfied for $g=qh+r_1$ and $\wtil{r}=-r_2$. In particular, $T_\om (qh+r_1)=sh+r$. Hence \eqref{DomRanIncl} holds.
\end{proof}

In the following lemma we determine a complement of $\wtil\cP$ in $\cP_{\deg(s)-1}$.

\begin{lemma}\label{L:tilP}
Let $\omega\in \Rat(\mathbb{T})$, say $\om=s/q$ with $s,q\in\cP$ co-prime. Define $\wtil\cP$ by \eqref{tilP} and set $\wtil\cQ=\cP_{\deg(s)-\deg(q)-1}$ if $\deg(s)>\deg(q)$ and $\wtil\cQ=\{0\}$ otherwise. Then $\cP_{\deg(s)-1} =\wtil\cP \dot{+} \wtil\cQ$, with $\dot{+}$ indicating a direct sum. In particular, $\cP_{\deg(s)-1} =\wtil\cP$ if and only if $\deg(s)\leq\deg(q)$.
\end{lemma}

\begin{proof}[\bf Proof]
For $\deg(s)\leq \deg(q)$ we have $\wtil\cQ=\{0\}$. Hence it is trivial that $\wtil\cP + \wtil\cQ$ is a direct sum. Also, in this case $\cP_{\deg(s)-1}\subset \cP_{\deg(q)-1}$ and consequently $s H^p+\cP_{\deg(q)-1}=s H^p+\cP_{\deg(s)-1}$, and this subspace of $H^p$ contains all polynomials. In particular, for any $r\in\cP_{\deg(s)-1}$ we have $q r\in s \cP_{\deg(q)-1} +\cP_{\deg(q)-1}$, which shows $r\in\wtil{\cP}$. Hence $\wtil{\cP}=\cP_{\deg(s)-1}$.

Next, assume $\deg(s)>\deg(q)$. Let $r\in\wtil\cQ$, i.e., $\deg(r)<\deg(s)-\deg(q)$. In that case $\deg(rq)<\deg(s)$ so that if we write $rq$ as $rq=r_1 s+r_2$ then $r_1\equiv 0$ and $r_2=rq$ with $\deg(rq)\geq \deg(q)$. Thus $rq$ is not in $s \cP_{\deg(q)-1}  + \cP_{\deg(q)-1}$ and, consequently, $r$ is not in $\wtil\cP$. Hence $\wtil\cP\cap\wtil\cQ=\{0\}$. It remains to show that $\wtil\cP + \wtil\cQ= \cP_{\deg(s)-1}$. Let $r\in\cP_{\deg(s)-1}$. Then we can write $rq$ as $rq=r_1 s +r_2$ with $\deg(r_1)<\deg(q)$ and $\deg(r_2)<\deg(s)$. Next write $r_2$ as $r_2=\wtil{r}_1 q+\wtil{r}_2$ with $\deg(\wtil{r}_2)<\deg(q)$. Since $\deg(r_2)<\deg(s)$, we have $\deg (\wtil{r}_1)<\deg(s)-\deg(q)$. Thus $\wtil{r}_1\in\wtil\cQ$. Moreover, we have
\[
rq=r_1 s +r_2= r_1 s +\wtil{r}_1 q+\wtil{r}_2= (r_1 s + \wtil{r}_2) + \wtil{r}_1 q,\mbox{ hence } (r-\wtil{r}_1)q=r_1 s + \wtil{r}_2.
\]
Thus $r-\wtil{r}_1\in\wtil\cP$, and we can write $r=(r-\wtil{r}_1)+\wtil{r}_1\in\wtil\cP+\wtil\cQ$.
\end{proof}

We now show that if $s$ has no roots on $\BT$, then the reverse inclusions in \eqref{DomRanIncl} also hold.

\begin{theorem}\label{T:DomRanIds}
Let $\omega\in \Rat(\mathbb{T})$, say $\om=s/q$ with $s,q\in\cP$ co-prime. Then $T_\om$ has closed range if and only if  $s$ has no roots on $\BT$, or equivalently, $s H^p+\wtil\cP$ is closed in $H^p$. In case $s$ has no roots on $\BT$, we have
\begin{equation}\label{DomRanId}
\Dom(T_\om)=qH^p+\cP_{\deg(q)-1}\ands
\Ran(T_\om)=s H^p+\wtil\cP.
\end{equation}
\end{theorem}

\begin{proof}[\bf Proof] The proof is divided into three parts.

\paragraph{\bf Part 1} In the first part we show that $s$ has no roots on $\BT$ if and only if $s H^p+\wtil\cP$ is closed in $H^p$. Note that for $\deg(s)\leq\deg(q)$ we have $\wtil\cP=\cP_{\deg(s)-1}$, and the claim coincides with Theorem \ref{T:main1}. For $\deg(s)>\deg(q)$, define $\wtil\cQ$ as in Lemma \ref{L:tilP}, viewed as a subspace of $H^p$. Since $\wtil\cQ$ is finite dimensional, $\wtil\cQ$ is closed in $H^p$. Hence, if $s H^p+\wtil\cP$ is closed, then so is $s H^p+\wtil\cP + \wtil\cQ=H^p+\cP_{\deg(s)-1}$. By Theorem \ref{T:main1}, the latter is equivalent to $s$ having no roots on $\BT$. Conversely, if $s$ has no roots on $\BT$, then $s H^p$ is closed, by Theorem 8 of \cite{V03} (see also the proof of Theorem \ref{T:main1}). Now using that $\wtil\cP$ is finite dimensional, and thus closed in $H^p$, it follows that $s H^p+\wtil\cP$ is closed.

\paragraph{\bf Part 2}
Now we show that $s H^p+\wtil\cP$ being closed implies \eqref{DomRanId}. In particular, this shows that $s$ having no roots on $\BT$ implies that $T_\om$ has closed range. Note that it suffices to show $\Dom(T_\om)\subset qH^p+\cP_{\deg(q)-1}$, since the equalities in \eqref{DomRanId} then follow directly from \eqref{DomRanIncl}. Assume $s H^p+\wtil\cP$ is closed. Then also $s H^p+\cP_{\deg(s)-1}$ is closed, as observed in the first part of the proof, and hence $s H^p+\cP_{\deg(s)-1}=H^p$. This also implies $s$ has no roots on $\BT$.

Write $s=s_- s_+$ with $s_-,s_+\in\cP$, with $s_-$ and $s_+$ having roots inside and outside $\BT$ only, respectively. Let $g\in\Dom (T_\om)$. Then $sg=qh+r$ for $h\in H^p$ and $r\in\cP_{\deg(q)-1}$. Note that $sH^p=s_- H^p$, since $s_+ H^p=H^p$.
By Theorem \ref{T:main1} we have $H^p=sH^p+\cP_{\deg(s_-)-1}$. Since $h\in H^p=sH^p+\cP_{\deg(s_-)-1}$, we can write $h=sh'+r'$ with $h'\in H^p$ and $r'\in \cP_{\deg(s_-)-1}$. Note that $\deg(qr'+r)<\deg(s_-q)$. We can thus write $qr'+r=r_1 s_- + r_2$ with $\deg(r_1)<\deg(q)$ and $\deg(r_2)<\deg(s_-)$. Then
\[
sg = qh+r=qsh'+qr'+r=qsh'+ r_1s_- +r_2=s(qh'+r_1s_+^{-1}) + r_2.
\]
Hence $r_2=s(g- qh'-r_1s_+^{-1})$. Since $\deg(r_2)<\deg(s_-)$, we can evaluate both sides (as well as the derivatives on both sides) at the roots of $s_-$, to arrive at $r_2\equiv 0$. Hence $s(g- qh'-r_1s_+^{-1})\equiv 0$. Dividing by $s$, we find $g=qh'+r_1s_+^{-1}$. Since $q$ and $s_+$ are co-prime and $r_1\in\cP_{\deg(q)-1}$, by Lemma \ref{L:invar} we have $r_1s_+^{-1}\in q H^p+\cP_{\deg(q)-1}$. Thus $g=qh'+r_1s_+^{-1}\in q H^p+\cP_{\deg(q)-1}$.

\paragraph{\bf Part 3}
In the last part we show that if $s$ has roots on $\BT$, then $T_\om$ does not have closed range. Hence assume $s$ has roots on $\BT$. Also assume $\Ran(T_\om)$ is closed. Since $sH^p +\wtil\cP \subset \Ran(T_\om)$ and $\Ran(T_\om)$ is closed, also $\overline{sH^p +\wtil\cP}\subset \Ran(T_\om)$. Since $\wtil\cQ$ is finite dimensional, and hence closed, $\overline{sH^p +\wtil\cP}+\wtil\cQ$ is closed and we have
\[
\overline{sH^p +\wtil\cP}+\wtil\cQ=\overline{sH^p +\wtil\cP+\wtil\cQ}=\overline{sH^p +\cP_{\deg(s)-1}}=H^p.
\]
Therefore, we have
\[
H^p=\overline{sH^p +\wtil\cP}+\wtil\cQ\subset \Ran(T_\om)+\wtil\cQ\subset H^p.
\]
It follows that $\Ran(T_\om)+\wtil\cQ=H^p$.

Let $h\in H^p$ such that $qh\not\in sH^p+\cP_{\deg(s)-1}$, which exists by Theorem \ref{T:main2}. Write $h=h'+r'$ with $h'\in\Ran(T_\om)$ and $r'\in\wtil\cQ$. Since $h'\in \Ran(T_\om)$, there exist $g\in H^p$ and $r\in\cP_{\deg(q)-1}$ such that
\[
sg=qh'+r=q(h-r')+r=qh -qr'+r.
\]
Write $r$ as $r=sr_1+r_2$ with $r_1,r_2\in\cP$, $\deg(r_2)<\deg(s)$. Note that $r'\in\wtil\cQ$, so that $\deg(qr')<\deg(s)$. Thus
\[
qh=sg+qr'-r=sg+qr'-sr_1-r_2=s(g-r_1)+(qr'-r_2)\in s H^p+\cP_{\deg(s)-1},
\]
in contradiction with $qh\not\in s H^p+\cP_{\deg(s)-1}$. Hence $\Ran(T_\om)$ is not closed.
\end{proof}

When $s$ has no roots on $\BT$ we have $\Ran(T_\om)=s H^p+\wtil\cP$ and thus, by Lemma \ref{L:tilP}, $\Ran(T_\om)+\wtil\cQ=H^p$. However, this need not be a direct sum in case $s$ has roots outside $\BT$. In the next lemma we obtain a different formula for $\Ran(T_\om)$, for which we can determine a complement in $H^p$.

\begin{lemma}\label{L:RanDirectSum}
Let $\omega\in \Rat(\mathbb{T})$, say $\om=s/q$ with $s,q\in\cP$ co-prime. Assume $s$ has no roots on $\BT$. Write $s = s_-s_+$ with the roots of $s_-$ and $s_+$  inside and outside $\mathbb{T}$, respectively. Define
\[
\wtil\cP_- = \{ r\in\cP \mid r q = \what r_1 s_- + \what r_2 \mbox{ for } \what r_1,\what r_2\in\mathcal{P}_{\deg(q)-1}\}
\]
and define $\wtil\cQ_-=\cP_{\deg(s_-)-\deg(q)-1}$ if $\deg(s_-)>\deg(q)$ and $\wtil\cQ_-=\{0\}$ if $\deg(s_-)\leq\deg(q)$. Then
\[
\Ran(T_\om)= s_-H^p \dot{+}\wtil\cP_-\ands \Ran(T_\om)\dot{+} \wtil\cQ_-= H^p.
\]
In particular, $\codim \Ran(T_\om)=\max\{0,\deg(s_-)-\deg(q)\}$.
\end{lemma}

\begin{proof}[\bf Proof]
It suffices to prove that $\Ran(T_\om)= s_-H^p +\wtil\cP_-$, that is, $s H^p+\wtil\cP= s_-H^p +\wtil\cP_-$, by Theorem \ref{T:DomRanIds}. Indeed, the direct sum claims follow since $H^p=s_- H^p\dot{+}\cP_{\deg(s_-)}$ is a direct sum decomposition of $H^p$, by Proposition \ref{P:directsum}, and $\cP_{\deg(s_-)}=\wtil\cP_-\dot{+}\wtil\cQ_-$ is a direct sum decomposition of $\cP_{\deg(s_-)}$, by applying Lemma \ref{L:tilP} with $s$ replaced by $s_-$.

We first show that  $s H^p+\wtil\cP\subset s_-H^p +\wtil\cP_-$. Let $f=sh+r$ with $h\in H^p$ and $r\in\wtil\cP$, say $rq=s r_1 +r_2$. Then $rq=s_-(s_+r_1)+r_2$. Now write $s_+r_1=q \wtil r_1+ \wtil r_2$ with $\deg(\wtil r_2)<\deg(q)$. Since $\wtil r_2$ and $r_2$ have degree less than $\deg(q)$ and
\[
q(r-s_-\wtil r_1)=s_-(s_+r_1)+r_2-q s_-\wtil r_1=s_- \wtil r_2 + r_2,
\]
it follows that $r-s_-\wtil r_1\in \wtil\cP_-$. Therefore
\[
f = sh+r= s_- (s_+h+ \wtil r_1)+(r-s_-\wtil r_1)\in s_- H^p+\wtil\cP_-.
\]
Thus $s H^p+\wtil\cP\subset s_-H^p +\wtil\cP_-$.

To complete the proof we prove the reverse implication. Let $f=s_- h+r\in s_- H^p+\wtil\cP_-$ with $h\in H^p$ and $r\in\wtil\cP_-$, say $rq= s_- \what r_1 + \what r_2$ with $\what r_1,\what r_2\in\cP_{\deg(q)-1}$. Set
\[
g=s_+^{-1}(qh+\what r_1)\in s_+^{-1}(q H^p+\cP_{\deg(q)-1})=q H^p+\cP_{\deg(q)-1},
\]
with the last identity following from Lemma \ref{L:invar}. Then $g\in \Dom(T_\om)$ and $T_\om g\in sH^p+\wtil\cP$. We show that $T_\om g=f$ resulting in $f\in sH^p+\wtil\cP$, as desired. We have
\begin{align*}
  sg &= s_- (qh+\what r_1)= s_- qh+ s_-\what r_1= s_- qh +rq -\what r_2\\
  &=q(s_- h + r)-\what r_2\in q H^p+\cP_{\deg(q)-1}.
\end{align*}
This proves $T_\om g=s_- h + r=f$, which completes our proof.
\end{proof}

Before proving Theorem \ref{T:Mainthm1a} we first give a few direct corollaries.

\begin{corollary}\label{C:RanCodim}
Let $\om\in\Rat(\BT)$ have no zeroes on $\BT$. Then
\begin{align*}
 & \codim \Ran(T_\om) = \\
 &\qquad = \max \left\{0,
\sharp\!\left\{
\begin{array}{l}\!\!\!
\textrm{zeroes of } \omega \textrm{ in }\mathbb{D} \textrm{ multi.}\!\!\!\!\\
\!\!\!\textrm{taken into account}\!\!\!
\end{array}\right\} - \sharp\!\left\{
\begin{array}{l}\!\!\!
\textrm{poles of }\omega \textrm{ multi.}\!\!\! \\
\!\!\!\textrm{taken into account}\!\!\!
\end{array}\right\}
\right \}.
\end{align*}
In particular, $T_\om$ is surjective if and only if the number of zeroes of $\om$ inside $\BD$ is less than or equal to the number of poles of $\om$ (all on $\BT$), in both cases with multiplicity taken into account.
\end{corollary}

\begin{corollary}\label{C:Fredholm}
Let $\om\in\Rat(\BT)$. Then $T_\om$ is Fredholm if and only if $\om$ has no zeroes on $\BT$. In that case the Fredholm index of $T_\om$ is given by
$$
\Index(T_\om) =
\sharp\left\{
\begin{array}{l}\!\!\!
\textrm{poles of }\omega \textrm{ multi.}\!\!\! \\
\!\!\!\textrm{taken into account}\!\!\!
\end{array}\right\} -
\sharp\left\{
\begin{array}{l}\!\!\!
\textrm{zeroes of } \omega \textrm{ in }\mathbb{D} \textrm{ multi.}\!\!\!\\
\!\!\!\textrm{taken into account}\!\!\!
\end{array}\right\}
.
$$
\end{corollary}

\begin{corollary}\label{C:InjSurj}
Let $\om\in\Rat(\BT)$ have no zeroes on $\BT$. Then $T_\om$ is either injective or surjective, and $T_\om$ is both injective and surjective if and only if the number of poles of $\om$ coincides with the number of zeroes inside $\BD$.
\end{corollary}

\begin{proof}[\bf Proof of Theorem \ref{T:Mainthm1a}]
Theorem \ref{T:Mainthm1a} follows by combining the various results from the present section. The claim the $T_\om$ is Fredholm if and only if $\om$ (or equivalently $s$) has no zeroes on $\BT$ along with the formula for the Fredholm index was given in Corollary \ref{C:Fredholm}, as a consequence of Theorem \ref{T:DomRanIds} and Lemma \ref{L:RanDirectSum}. The formula for $\kernel(T_\om)$ in \eqref{DRK} follows from Lemma \ref{L:kernel2}, noting that $s_0\equiv 1$, and the formulas for $\Dom(T_\om)$ and $\Ran(T_\om)$ follow from Theorem \ref{T:DomRanIds}. The formula for a complement of $\Ran(T_\om)$ is obtained in Lemma \ref{L:RanDirectSum}, and, finally, the claims regarding injectivity and surjectivity of $T_{\om}$ are listed in Corollary \ref{C:InjSurj}.
\end{proof}

\section{Fredholm properties of $T_\omega$: General case}\label{S:Fredholm2}

In this section we prove Theorem \ref{T:Mainthm1} in the general case, i.e., for $\om\in\Rat$. In order to do this we need some preliminary results, which are closely connected to non-canonical Wiener-Hopf factorization.

\begin{lemma}\label{L:factor}
Let $\om\in\Rat$ and denote by $\kappa=l^+ -l^-$ the difference between the number $l^+$ of zeroes of $\om$ in
$\BD$ and the number $l^-$ of poles of $\om$ in $\BD$.
Then we can write $$\om(z) = \om_-(z) (z^\kappa\om_0(z)) \om_+(z)$$ where $\om_-$ has no poles or zeroes outside ${\BD}$, $\om_+$ has no poles or zeroes inside $\overline{\BD}$ and $\om_0$ has all its poles and zeroes on $\BT$, i.e. $\om_0\in\Rat(\BT)$. The functions $\om_-,\om_0, \om_+$ are unique up to a multiplicative constant. In this case we have
$$ T_\om = T_{\om_-} T_{z^\kappa\om_0} T_{\om_+}.$$
\end{lemma}

\begin{proof}[\bf Proof]
Suppose that $\om =s/q$ with $s,q\in\cP$ co-prime, and let $s = s_- s_0 s_+$ where $s_-$ is monic and has all its roots in $\BD$, $s_0$ has all its roots on $\BT$ and $s_+$ has all its roots outside $\overline{\BD}$. Let $q = q_- q_0 q_+$ be similarly defined, i.e. $q_-$ monic with all its roots in $\BD$, $q_0$ has all its roots on $\BT$ and $q_+$ has all its roots outside $\overline{\BD}$. Let $t_j, j=1\ldots l^+$ be the roots of $s_-$ and $\tau_j, j= 1,\ldots l^-$ be the roots of $q_-$, possibly with repetitions. Then we can write
$$\frac{s_-}{q_-} =\frac{\Pi_{j=1}^{l^+}(z- t_j ) }{\Pi_{j=1}^{l^-} (z - \tau_j )}= z^\kappa \frac{\Pi_{j=1}^{l^+}(1 - t_j z^{-1}) }{\Pi_{j=1}^{l^-} (1 - \tau_j z^{-1})}$$
where $\kappa = l^+ - l^-$.
Put $\om_0 = \frac{s_0}{q_0}$,
$$\om_- =\frac{1}{z^\kappa }\frac{s_-}{q_-}=\frac{1}{z^\kappa }\frac{\Pi_{j=1}^{l^+}(z- t_j ) }{\Pi_{j=1}^{l^-} (z - \tau_j )}$$
and $\om_+ = \frac{s_+}{q_+}$.
Then $\om_-$ has no zeroes or poles outside $\overline{\BD}$ including infinity, as $\lim_{z\to\infty}\om_-(z)=1$, $\om_+$ has no poles and zeroes inside $\overline{\BD}$, $\om_0\in\Rat(\BT)$
and we have the desired factorization $\om=\om_-(z^\kappa\om_0)\om_+$.

The uniqueness may be seen as follows: clearly $\kappa $ is uniquely determined by $\om$.
Suppose $\om_-^\prime\om_0^\prime\om_+^\prime=\om_-\om_0\om_+$. Then
$(\om_-^\prime)^{-1}\om_-\om_0=\om_0^\prime\om_+^\prime(\om_+)^{-1}$, and it follows that this is a function in $\Rat(\BT)$. It is then easily seen that there are constants $c_0,c_-,c_+$ such that
$\om_0^\prime =c_0\om_0$, $\om_-^\prime=c_-\om_-$ and $\om_+^\prime=c_+\om_+$, with
$c_-c_0c_+=1$.

Note that $f\in\Dom (T_{\om_-} T_{z^\kappa\om_0} T_{\om_+} )$ if and only if
\[
T_{\om_+}f=\om_+ f  \in \Dom  (T_{\om_-}T_{z^\kappa\om_0})=\Dom (T_{z^\kappa\om_0}).
\]
Now let $f\in\Dom (T_{\om_-} T_{z^\kappa\om_0} T_{\om_+} )$.
So there are $h\in L^p$ and $\rho\in\Rat_0(\BT)$ with $z^\kappa\om_0(\om_+f) = h + \rho$ and $T_{z^\kappa\om_0} \om_+f=\mathbb{P}h$.
Furthermore,
\[
\om f  = \om_- (z^\kappa\om_0\om_+ f) = \om_- (h + \rho)  = \om_- h + \om_-\rho .
\]
Now $\om_-\rho$ is a rational function which has poles only in the closed unit disc. Moreover,
as $\lim_{z\to\infty}\om_-(z)=1$ and $\rho\in\Rat_0(\BT)$ we have that $\omega_-\rho$ is strictly proper. Hence, by Lemma \ref{L:reducRatT}, we can write
$\om_-\rho$ uniquely as $\om_-\rho = g + \rho'$ with $g$ a rational function with poles only inside $\BD$ and $\rho'\in\Rat_0(\BT)$. Then also $g$ is a strictly proper rational function, as both
$\om_-\rho$ and $\rho'$ are strictly proper. We conclude that
\[
\om f =(\om_- h +g)+\rho'
\]
and since $\om_-h+g\in L^p$
we have $f\in \Dom (T_\om)$ and $T_\om f=\mathbb{P}(\om_-h+g)$.
Now since $g$ is a rational function which is strictly proper and it has all its poles in $\BD$, $g$ has a realization
$g(z)=c(z-A)^{-1}b$, with $A$ a stable matrix. Then $g(z)=\sum_{j=0}^\infty \frac{cA^jb}{z^{j+1}}$, and hence $\mathbb{P}g=0$. Thus we see that $f\in\Dom(T_\om)$ and $T_\om f=\mathbb{P}(\om_-h).$

On the other hand
\[
T_{\om_-}T_{z^\kappa\om_0}T_{\om_+}f=T_{\om_-}T_{z^\kappa\om_0}\om_+f
=T_{\om_-}(\mathbb{P}h)=\mathbb{P}(\om_-\mathbb{P}h).
\]
Write $h=h_-+h_+$, where $h_+=\mathbb{P}h$. Then $\BP(\om_-h_-)=0$ since both $\om_-$ and $h_-$ are anti-analytic. Thus $\om_-h=\om_-h_-+\om_-h_+$
and $\mathbb{P}(\om_-h)=\mathbb{P}(\om_-h_+)$. This implies that
\[
T_\om=T_{\om_-}T_{z^\kappa\om_0}T_{\om_+},
\]
provided $\Dom (T_\om)$ is equal to $\Dom (T_{\om_-} T_{z^\kappa\om_0} T_{\om_+})$.

We already proved that $\Dom (T_{\om_-} T_{z^\kappa\om_0} T_{\om_+})\subset \Dom (T_\om)$. To prove the reverse inclusion, suppose that $f\in \Dom (T_\om)$. Then there are $g\in L^p$ and $\rho\in\Rat_0(\BT)$ with $\om f = g + \rho$. Since $\om_-^{-1}\rho\in\Rat$ has poles only in $\BD$, by Lemma \ref{L:reducRatT} we can write $\om_-^{-1}\rho$ uniquely as $\om_-^{-1}\rho =  g' + \rho'$ with $g'$ strictly proper and with poles only in $\BD$ and $\rho'\in\Rat_0(\BT)$. Also, $\om_-^{-1} g\in L^p$ because $\om_-^{-1}$ has no poles on $\BT$. But then $z^\kappa\om_0\om_+ f = \om_-^{-1} g + \om_-^{-1}\rho = (\om_-^{-1} g + g') + \rho'$ is in $L^p+\Rat_0(\BT)$. Hence $\om_+ f \in\Dom(T_{z^\kappa\om_0})$, which implies $f\in \Dom(T_{\om_-} T_{z^\kappa\om_0} T_{\om_+} )$.
\end{proof}

\begin{remark}
Compare this with Theorem 16.2.3 of \cite{GGK02} and Proposition 2.14 of \cite{BS06} from which it follows that if $\overline{a}, b \in H^\infty$, $c\in L^\infty$ then $T_{abc} = T_a T_c T_b$.
\end{remark}

Observe that $T_{\om_-}$ and $T_{\om_+}$ are bounded and have a bounded inverse, in fact
$T_{\om_-}^{-1}=T_{\om_-^{-1}}$ and $T_{\om_+}^{-1}=T_{\om_+^{-1}}$. Hence the Fredholm properties of $T_\om$ are the same as the Fredholm properties of $T_{z^\kappa\om_0}$.
If $\kappa \geq 0$, then these properties are described by the results of the previous section. It remains to study the case where $\kappa <0$. For this case we have the following lemma.

\begin{lemma}\label{L:KappaNeg}
Let $\om\in\Rat(\BT)$ and $\kappa <0 $. Then $T_{z^{\kappa}\om} = T_{z^{\kappa}}T_\om $.
Moreover, $T_{z^{\kappa}}T_\om$ is Fredholm if and only if $T_\om T_{z^{\kappa}}$ is Fredholm.
\end{lemma}

\begin{proof}[\bf Proof]
Let $\om=\frac{s}{q}$, where $s$ and $q$ are coprime and $q$ has all its roots on $\BT$.
First we show that $T_{z^\kappa\om}=T_{z^\kappa}T_\om$.

Let $f\in \Dom (T_\om)$, then $\om f=h+\phi$, with $h\in L^p$ and $\phi\in \Rat_0(\BT)$. Then
$z^\kappa\om f=z^\kappa h+z^\kappa\phi$. Clearly $z^\kappa h\in L^p$. Write $z^\kappa \phi =g'+\phi'$, where $g'$ is rational, strictly proper and has a pole only at zero (recall, $\kappa <0$), and $\phi'$ is in $\Rat_0(\BT)$; see
Lemma 2.4. Then $z^\kappa\om f= (z^\kappa h+g')+\phi'$, and hence $f\in \Dom(T_{z^\kappa\om})$. This shows $\Dom(T_{z^\kappa}T_\om)=\Dom (T_\om)\subset \Dom(T_{z^\kappa\om})$.

Conversely, if $f\in \Dom(T_{z^\kappa\om})$ then there is a $g\in L^p$ and a $\rho\in \Rat_0(\BT)$ such that $z^\kappa \om f =g+\rho$ and $T_{z^\kappa\om}f=\BP g$. Then $\om f= z^{-\kappa} g +z^{-\kappa}\rho$. Since $\kappa <0$ and $\rho\in\Rat_0(\BT)$, we have $z^{-\kappa}\in\cP$ and, by Euclidean division, we can write $z^{-\kappa}\rho= r+\psi$ with $r\in\cP_{-\kappa-1}$ and $\psi\in\Rat_0(\BT)$. Clearly $z^{-\kappa}g\in L^p$. Thus $\om f= (z^{-\kappa} g +r) +\psi$ is in $L^p+\Rat_0(\BT)$. Hence $f\in\Dom(T_{\om})=\Dom(T_{z^\kappa}T_\om)$ and $T_\om=\BP z^{-\kappa} g +r$. In particular,  this implies $\Dom(T_{z^\kappa}T_\om)=\Dom(T_{z^\kappa\om})$.

To complete the proof of the first claim, it remains to show that
\[
\BP g=T_{z^\kappa\om}f= T_{z^\kappa}T_\om f= T_{z^\kappa} (\BP z^{-\kappa} g +r)= \BP z^\kappa (\BP z^{-\kappa} g +r).
\]
Since $\deg (r)< -\kappa$, we have $\BP z^\kappa r=0$. Thus we have to show that $\BP g=\BP z^\kappa \BP z^{-\kappa} g$. Write $g(z)=\sum_{j=-\infty}^\infty z^j g_j$. Then $\mathbb{P} z^{-\kappa} g=\sum_{j=\kappa}^\infty g_j z^{j-\kappa}$. Since $\kappa<0$, we have
\[
\BP z^\kappa \mathbb{P} z^{-\kappa} g
=\BP \left(\sum_{j=\kappa}^\infty g_j z^{j}\right)
=\sum_{j=0}^\infty g_j z^{j} =\BP g,
\]
which finalizes the proof of the claim that $T_{z^\kappa\om}= T_{z^\kappa}T_\om$.

To prove the second part of the statement, we show that the difference $T_{z^\kappa}T_\om -T_\om T_{z^\kappa}$ is a bounded finite rank operator, from which the result follows. More specifically, we show that $T_{z^\kappa}T_\om -T_\om T_{z^\kappa}$ is zero on $z^{-\kappa} H^p$. Note that $z^{-\kappa} \Dom(T_\om)$ is dense in $z^{-\kappa} H^p$, since $\Dom(T_\om)$ is dense in $H^p$. Thus it suffices to show that $T_{z^\kappa}T_\om f = T_\om T_{z^\kappa}f$ for all $f=z^{-\kappa} g\in z^{-\kappa} \Dom(T_\om)$; note that by the last claim of Lemma \ref{L:domain} we have $z^{-\kappa} \Dom(T_\om)\subset \Dom(T_\om)$ since $\kappa<0$.

Thus, let $f=z^{-\kappa} g\in z^{-\kappa} \Dom(T_\om)$, say $s g= q h+r$ with $h\in H^p$ and $\deg(r)<\deg(q)$. Note that $T_{z^\kappa} f= g$, so that $T_\om T_{z^\kappa}f=T_{\om} g=h$. On the other hand, $sf= q z^{-\kappa} h  + z^{-\kappa}r =q (z^{-\kappa} h+r_2)+r_1$, where $r_2,r_1\in \cP$, $\deg(r_2)<-\kappa$, $\deg(r_1)<\deg(q)$ are such that $z^{-\kappa}r= q r_2+r_1$. Then $T_{\om} f=z^{-\kappa} h+r_2$, which shows $T_{z^{-\kappa}} T_{\om} f= \BP ( h+  z^{\kappa} r_2)= h$, since $\deg(r_2)<-\kappa$. Thus $T_{z^\kappa}T_\om f = T_\om T_{z^\kappa}f$, as claimed, and the proof is complete.
\end{proof}

\begin{proof}[\textbf{Proof of the Fredholm claim of Theorem \ref{T:Mainthm1}}]
Let $\om = s/q\in\Rat(\BT)$ where $s$ and $q$ are coprime. Put $\om = \om_- z^\kappa\om_0 \om_+$ as in Lemma \ref{L:factor} above, where $\om_+$ has all its poles and zeroes outside $\overline{\BD}$, $\om_0 \in \Rat(\BT)$ and $\om_-$ has all its poles and zeroes in $\mathbb{D}$. Then $T_\om = T_{\om_-} T_{z^\kappa\om_0} T_{\om_+}$. Clearly $T_{\om_-}$ and $T_{\om_+}$ are boundedly invertible, so $T_\om$ is Fredholm if and only if $T_{z^\kappa\om_0} $ is Fredholm.

If $\kappa \geq 0$ it follows from Corollary \ref{C:Fredholm} that $T_{z^\kappa\om_0}$ is Fredholm if and only if $\om_0$ has no zeroes on $\BT$. This proves the Fredholm claim of Theorem \ref{T:Mainthm1} for the case where $\kappa \geq 0$.

Now let $\kappa <0$. Then $T_{z^\kappa\om_0}=T_{z^\kappa}T_{\om_0}$, by Lemma \ref{L:KappaNeg}. Suppose first that $\om_0$ has no zeroes on $\BT$, so that $T_{\om_0}$ is Fredholm by Corollary \ref{C:Fredholm}. As $T_{z^\kappa}$ is Fredholm as well, $T_{z^\kappa}T_{\om_0}$ is Fredholm.
Conversely, assume $T_{z^\kappa\om_0}$ is Fredholm. Then $T_{z^\kappa}T_{\om_0}=T_{z^\kappa\om_0}$ is Fredholm, and thus  $T_{\om_0}T_{z^\kappa}$ is Fredholm, again using Lemma \ref{L:KappaNeg}. Now $T_{z^\kappa}T_{z^{-\kappa}}=I$, and $T_{z^{-\kappa}}$ is Fredholm. Hence (see \cite{G66}, Theorem IV.2.7) $T_{\om_0}T_{z^\kappa}T_{z^{-\kappa}}=T_{\om_0}$ is Fredholm. By Corollary \ref{C:Fredholm} again, this implies that $\om_0$ has no zeroes on $\BT$, and hence also $\om$ has no zeroes on $\BT$.
\end{proof}

For $\omega\in L^\infty$ we have the following result by L. A. Coburn (see \cite{BS06} Theorem 2.38): {\em
If $\omega\in L^\infty$ and $\omega$ does not vanish identically then either the kernel of $T_\omega$ in $H^p$ is trivial or $T_\omega$ has dense range in $H^p$.}

For $\omega\in\textup{Rat}$ with poles in $\mathbb{T}$ the theorem of Coburn does not hold in full generality but we do have the following, which also proves the second part of Theorem \ref{T:Mainthm1}, i.e., the statement on the index.

\begin{theorem}\label{T:index}
Let $\om\in\Rat$ with possibly poles on $\BT$. If $T_\om$ is Fredholm then
\[
\Index (T_\om) = \sharp \left\{\begin{array}{l}\!\!\!
 \textrm{poles of } \om \textrm{ in }\overline{\BD} \textrm{ multi.}\!\!\! \\
\!\!\!\textrm{taken into account}\!\!\!
\end{array}\right\}  -
\sharp \left\{\begin{array}{l}\!\!\! \textrm{zeroes of } \om\textrm{ in }\BD  \textrm{ multi.}\!\!\! \\
\!\!\!\textrm{taken into account}\!\!\!
\end{array}\right\} .
\]
and $T_\om$ is invertible if and only if $\Index (T_\om) = 0$, i.e.
\[
\sharp \left\{\begin{array}{l}\!\!\!
 \textrm{poles of } \om\textrm{ in }\overline{\BD} \textrm{ multi.}\!\!\! \\
\!\!\!\textrm{taken into account}\!\!\!
\end{array}\right\}  =
\sharp \left\{\begin{array}{l}\!\!\! \textrm{zeroes of } \om\textrm{ in }\BD  \textrm{ multi.}\!\!\! \\
\!\!\!\textrm{taken into account}\!\!\!
\end{array}\right\} .
\]
Furthermore, if $T_\om$ is Fredholm then $T_\om$ is either injective or surjective.
\end{theorem}

\begin{proof}[\bf Proof]
If $T_\om$ is Fredholm then $\om$ has no zeroes on $\BT$. Let $\om = \om_-z^\kappa \om_0 \om_+$ be the factorization of $\om$ as in Lemma \ref{L:factor}. Then $T_\om=T_{\om_-} T_{z^\kappa\om_0}T_{\om_+}$.
As $T_{\om_-}$ and $T_{\om_+}$ are bounded and invertible, it follows that $\Index (T_\om)=\Index (T_{z^\kappa\om_0})$.
If $\kappa \geq 0$ then by Corollary \ref{C:Fredholm}
$$
\Index (T_{z^\kappa\om_0})=
\sharp\left\{\begin{array}{l}\!\!\!
\textrm{poles of }\omega_0 \textrm{ multi.}\!\!\! \\
\!\!\!\textrm{taken into account}\!\!\!
\end{array}\right\} -\kappa
.
$$
Since $\kappa=l^+-l^-$ is the difference between the number of zeroes of $\om$ in the open unit disc and the number of poles of $\om$ in the open unit disc we see that
$$
\Index (T_\om) =  \sharp \left\{\begin{array}{l}\!\!\!
 \textrm{poles of } \om\textrm{ in }\overline{\BD} \textrm{ multi.}\!\!\! \\
\!\!\!\textrm{taken into account}\!\!\!
\end{array}\right\}  -
\sharp \left\{\begin{array}{l}\!\!\! \textrm{zeroes of } \om\textrm{ in }\BD  \textrm{ multi.}\!\!\! \\
\!\!\!\textrm{taken into account}\!\!\!
\end{array}\right\}
$$
as stated.

If $\kappa <0$ then, as observed in the proof of the previous lemma, $T_{z^\kappa\om_0} = T_{z^\kappa}T_{\om_0}= T_{\om_0} T_{z^\kappa}+\Psi$ for some bounded $\Psi$ of finite rank. By \cite{G66}, Theorem V.2.1 we have
$
\Index (T_{z^\kappa\om_0})  =\Index (T_{\om_0} T_{z^\kappa})
$,
and by \cite{G66}, Theorem IV.2.7 this is equal to $\Index (T_{\om_0}) +\Index (T_{z^\kappa})=
\Index (T_{\om_0}) -\kappa$.
From here on the proof is the same as in the case $\kappa \geq 0$.

Clearly, in case $T_\om$ is Fredholm, $T_\om$ is injective if and only if $T_{z^\kappa \om_0}$ is injective, and similarly $T_\om$ is surjective if and only if $T_{z^\kappa\om_0}$ is surjective. In case $\kappa \geq 0$ then $z^\kappa\om_0\in\Rat(\BT)$ and from Corollary \ref{C:InjSurj} it follows that $T_\om$ is either injective or surjective. On the other hand let $\kappa < 0$. For $h\in H^p$, $f=(z^\kappa\om_0)^{-1}h\in\Dom(T_{z^\kappa\om_0})$ (recall that $\om_0$ has no zeroes on $\BT$ as $T_\om$ is Fredholm) with $T_{z^\kappa\om_0}f = h$ showing that $T_{z^\kappa\om_0}$ is surjective. In addition, $T_{z^\kappa\om_0}$ is not injective as $\{z^j:j < M\}\subset \kernel(T_{z^\kappa\om_0})$ where $M = \kappa + \sharp\{\textrm{poles of } \om_0\}$.
\end{proof}

We conclude this section with a characterization of invertibility of $T_\om$ and a formula for the inverse when it exists. Here invertibility means that $T_\om$ is bijective, so that the inverse is bounded.
The classical result for continuous symbols is that $T_\om$ is invertible if and only if $\om$ has no zeroes on $\BT$ and $T_\om$ is Fredholm of index zero; the inverse is then provided using the factors in the Wiener-Hopf factorization \cite[Theorem XVI.2.2]{GGK02}.

\begin{proposition}\label{P:inverse}
Let $\om\in\Rat$ with at least one pole on $\BT$ and let $\kappa$ be the difference between the number of zeroes of $\om$ in $\BD$ and the number of poles of $\om$ in $\BD$. Then $T_\om$ is invertible if and only if $\om$ has no zeroes on $\BT$ and
$\kappa$ is also equal to the number of poles of $\om$ on $\BT$.
In that case $\om$ factorizes as
$$
\om(z)=\om_-(z)\frac{z^\kappa}{q_0(z)}\om_+(z),
$$
where $\om_-$ has no poles or zeroes outside $\BD$, $\om_+$ has no poles or zeroes inside
$\overline{\BD}$ and $q_0$ is a polynomial of degree $\kappa$ with all its roots on $\BT$, and moreover,
$$
T_\om^{-1}=T_{\om_+^{-1}}T_{\frac{q_0}{z^\kappa}}T_{\om_-^{-1}}.
$$
\end{proposition}

\begin{proof}[\bf Proof.]
Let $\om=\om_-(z^\kappa\om_0)\om_+$ be the factorization of
Lemma \ref{L:factor}. Then $T_{\om_-}$ and $T_{\om_+}$ are invertible with
inverses $T_{\om_-^{-1}}$ and $T_{\om_+^{-1}}$, and thus it is seen that the inverse of
$T_\om$ exists
if and only if the inverse of $T_{z^\kappa \om_0}$ exists. Since an invertible operator is certainly Fredholm with index zero, it follows that $\om_0$ has no zeroes on $\BT$, and so $\om_0=
1/q_0$ for some polynomial $q_0$ with roots only on the unit circle. The Fredholm index being zero implies that $\kappa=\deg(q_0)$. Conversely, if $\om_0$ is of this form, then $T_{z^\kappa \om_0}$ is one-to-one by Corollary \ref{C:kerdim} and onto by Corollary \ref{C:RanCodim}.

It remains to show the formula for the inverse, and here too it suffices to show that
$$
T_{\frac{z^\kappa}{q_0}}^{-1}= T_{\frac{q_0}{z^\kappa}}.
$$
Note that $T_{\frac{q_0}{z^\kappa}}$ is a bounded operator. To show that this is the inverse
of $T_{\frac{z^\kappa}{q_0}}$, let $f\in H^p$ and write $q_0(z)f(z)=z^\kappa h(z)+r(z)$ where
$h\in H^p$ and $r$ is a polynomial with $\deg(r) <\kappa$. Then $T_{\frac{q_0}{z^\kappa}}f=h$,
and on the other hand, from $q_0(z)f(z)=z^\kappa h(z)+r(z)$ we have that $h\in \Dom (T_{\frac{z^\kappa}{q_0}})$ with $T_{\frac{z^\kappa}{q_0}} h=f$.
\end{proof}

\section{Matrix representation}\label{S:matrix}

For $n\in\BZ$, let $e_n$ be the function $e_n(z) = z^n, z\in\mathbb{T}$. Then $\{e_n\}_{n = 0}^\infty$ is the standard basis for $H^p$. Where convenient, we shall denote $e_n$ simply by $z^n$.


Now let $\omega\in\Rat$ with possibly poles on $\mathbb{T}$. Since the polynomials $\cP$ are contained in the domain of the closed operator $T_\om$ defined in \eqref{Toeplitz}, by inspecting the action of $T_\om$ on the monomials $z^n$ and expressing the result as a power series, it is possible to determine a matrix representation $[T_\omega]$ of the operator $T_\om(H^p\to H^p)$ with respect to the basis $\{e_n\}_{n = 0}^\infty$.

In this section we shall prove Theorem  \ref{T:Mainthm2}, which states that this matrix representation $[T_\om]$ has a Toeplitz structure, i.e., $[T_\om] = [a_{m-n}]_{m,n=0}^\infty$ for some sequence $(a_n)_{n\in\mathbb{Z}}$. Here $a_n$ has a polynomial bound, $a_n = O(n^j)$ for some $j\in\mathbb{N}$.

We first prove the following lemma, which is an explicit formulation of the Euclidean algorithm for
dividing $z^N - 1$  by $(z-1)^m$, where $m < N$.

\begin{lemma}\label{L:division}
For any natural number $N$ and any $m<N$
\begin{equation}\label{special_euclid}
z^N-1=(z-1)^m \sum_{i=0}^{N-m}
\begin{pmatrix} i+m-1 \\ m-1 \end{pmatrix}
z^{N-m-i} +
\sum_{j=0}^{m-1}\begin{pmatrix} N \\ j \end{pmatrix}  (z-1)^j .
\end{equation}
\end{lemma}

\begin{proof}[\bf Proof]
The proof is by induction on $m$. The case $m=1$ is just the well-known formula
$$
z^N-1=(z-1)(z^{N-1} + z^{N-2}+\cdots + z +1).
$$
For $m>1$, to show that \eqref{special_euclid} holds, we have to prove the following:
\begin{equation}\label{to_show}
\sum_{i=0}^{N-m}\!
\begin{pmatrix} i+m-1 \\ m-1 \end{pmatrix}
z^{N-m-i}= (z -1)\! \! \! \sum_{i=0}^{N-m-1}\!
\begin{pmatrix} i+m \\ m \end{pmatrix}
z^{N-m-i-1} +\begin{pmatrix} N \\ m \end{pmatrix}.
\end{equation}
To see this, we shall make use of the so-called hockey-stick formula, which implies that
$\displaystyle\sum_{i=0}^{N-m} \begin{pmatrix} i+m-1 \\ m-1 \end{pmatrix}  =\begin{pmatrix} N \\ m \end{pmatrix}$. Thus, the right hand side of \eqref{to_show} is equal to
\begin{equation}\label{first_step}
 (z -1)\sum_{i=0}^{N-m-1}
\begin{pmatrix} i+m \\ m \end{pmatrix}
z^{N-m-i-1} +\sum_{i=0}^{N-m} \begin{pmatrix} i+m-1 \\ m-1 \end{pmatrix}.
\end{equation}
Note that the remainder of $\displaystyle\sum_{i=0}^{N-m}
\begin{pmatrix} i+m-1 \\ k-1 \end{pmatrix}
z^{N-m-i}$ upon dividing by $z -1$ is equal to $\displaystyle\sum_{i=0}^{N-m} \begin{pmatrix} i+m-1 \\ k-1 \end{pmatrix}$, so the remainder term in \eqref{to_show} is correct.

To finish the proof it remains to compare the coefficients of $z^k$ on the left and right hand sides of
\eqref{to_show} with $k\geq 1$. A straightforward rewriting of the right hand side shows that the equality \eqref{to_show} follows from the basic property of binomial coefficients.
\end{proof}


\begin{example}
Let $\om (z) = (z - 1)^{-m}$, i.e., $s\equiv 1$ and $q(z)=(z - 1)^{m}$. From Proposition \ref{P:welldefclosed} we know $qH^p + \mathcal{P}_{m-1} \subset \Dom (T_\omega)$ which contains all the polynomials. Put
$$a_{-i} = \begin{pmatrix} i+m-1 \\ m-1 \end{pmatrix}, i=0,1,2,\dots $$ and
\[
b_{-j} =
\left\{ \begin{array}{ll}
0 & \qquad j < m \\
a_{m-j} & \qquad j \geq m
\end{array} \right . .
\]
From Lemma \ref{L:division} above, for $N > m$ we can write
\[
\begin{array}{ll}
z^N   & = \displaystyle(z - 1)^m \sum_{i=0}^{N-m}\begin{pmatrix} i+m-1 \\ m-1 \end{pmatrix}z^{N-m-i} + \sum_{j=0}^{m-1}\begin{pmatrix} N \\ j \end{pmatrix}  (z-1)^j + 1 \\
& = \displaystyle q(z) \sum_{i=0}^{N-m}a_{-i}z^{N-m-i} + \sum_{j=0}^{m-1}\begin{pmatrix} N \\ j \end{pmatrix}  (z-1)^j + 1 \\
& = \displaystyle q(z)\sum_{j=0}^{N-m}a_{-(N-m-j)}z^j + \sum_{j=0}^{m-1}\begin{pmatrix} N \\ j \end{pmatrix}  (z-1)^j + 1.
\end{array}
\]
Put $$r(z) = \sum_{j=0}^{m-1}\begin{pmatrix} N \\ j \end{pmatrix}  (z-1)^j + 1.$$
Then $r$ is a polynomial with $\deg (r) < m = \deg q$ and since $s\equiv 1$,
\[
s(z)z^N = q(z)\sum_{j=0}^{N-m}a_{-(N-m-j)}z^j + r
\]
and so from Lemma \ref{L:domain}, for $N>m$ we have
\[
T_\om z^N = \sum_{j=0}^{N-m}a_{-(N-m-j)}z^j = \sum_{j=0}^{N-m}b_{-(N-j)}z^j = \sum_{j=0}^{N}b_{-(N-j)}z^j
\]
since $b_{-j} = 0$ for $j < m$. From Lemma \ref{L:kernel2} we have  $\kernel(T_\om) = \{z^j, j < m \}$ and so the matrix representation of $T_\om$ will be an upper triangular Toeplitz matrix with the first $m$-columns zero.
\end{example}

\begin{proposition}\label{P:matrix}
Let $\om\in\Rat_0(\BT)$, say $\om=s/q$ with  $s,q\in\cP$ co-prime, $q$ having all its roots in $\BT$ and $\deg(s) = n < m = \deg(q)$. Then, for $N \geq m-n$,
$$ T_\omega z^N = \sum_{j=1}^{N-m+n+1} a_{-j} z^{N-m+n+1-j}$$
where $a_{-j} = O(j^{M-1})$ with $M$ the maximum of the orders of the poles of $\om$ on $\BT$. Thus the matrix representation $[T_\omega]$ of $T_\omega$ with respect to the standard basis $\{z^n\}_{n=0}^\infty$ of $H^p$ is given by
$$\begin{array}{ll}
[T_\omega] =&
\left (
\begin{array}{cccccc}
0 \cdots 0 & a_{0} & a_{-1} & a_{-2} & a_{-3} & \cdots \\
0 \cdots 0 & 0 & a_{0} & a_{-1} & a_{-2} &  \cdots \\
0 \cdots 0 & 0 & 0 & a_{0} & a_{-1} &  \cdots \\
\vdots & \multicolumn{5}{c}{\ddots} \\
\end{array}
\right )\\
&\ \ \quad \underbrace{\qquad}_{m-n}
\end{array}
$$
\end{proposition}

\begin{proof}[\bf Proof]
Since $\om\in\Rat_0(\BT)$, by Corollary \ref{C:Rat0Inject} we have $\cP_{m-n-1} \subset \kernel (T_\om)$.

Suppose  $q(z) = \prod_{j=1}^t (z - \alpha_j)^{m_j}$ with $\al_j\in\BT$ the poles of $\om$ with multiplicities $m_j$. By partial fractions decomposition we can write $\omega  = \sum_{j=1}^t \omega_j$, where
$$
\omega_j (z)
 =  \frac{s_j (z)}{(z - \alpha_j)^{m_j}}, \quad \mbox{for } s_j\in\cP_{m_j-1},\ j=1\cdots t.
$$

Note that $[T_\om] = \sum_{j=1}^t [T_{\om_j}]$ and if $s_j(z) = c_0 + c_1 z + c_2 z^2 + \cdots + c_{m_j-1}z^{m_j-1}$ then $[T_{\om_j}] = \sum_{i=0}^{m_j-1} c_i[T_{z^i/(z - \alpha_j)^{m_j}}]$. From this it follows that it suffices to prove the result for $\om (z) = z^n / (z-\alpha)^m$. To this end, assume $\om = s/q$ where $s(z ) = z^n$ and $q(z ) = (z - \alpha)^m$ for some $\alpha\in\BT$.
By Lemma \ref{L:division} we have
\[
z^N = (z - 1)^m \sum_{j=0}^{N-m}
\begin{pmatrix}
j + m - 1\\
m-1
\end{pmatrix}
z^{N-m-j}
+
 \sum_{j=1}^{m-1}
\begin{pmatrix}
N\\
j
\end{pmatrix}
(z - 1)^j + 1.
\]
By replacing $z$ with $\frac{z}{\alpha}$ we can write
\[
\left(\frac{z}{\alpha} \right )^N
=\left (\frac{z}{\alpha} - 1\right )^m \sum_{j=0}^{N-m}
\begin{pmatrix}
j + m - 1\\
m-1
\end{pmatrix}
\left(\frac{z}{\alpha}\right)^{N-m-j}
+  \sum_{j=1}^{m-1}
\begin{pmatrix}
N\\
j
\end{pmatrix}
\left (\frac{z}{\alpha} - 1\right )^j + 1.
\]
Multiplying with $\alpha^N$ results in
\[
z^N
=(z - \alpha)^m \sum_{j=0}^{N-m}
\begin{pmatrix}
j + m - 1\\
m-1
\end{pmatrix}
\alpha^{j-1}z^{N-m-j}
+\sum_{j=1}^{m-1}
\begin{pmatrix}
N\\
j
\end{pmatrix}
\alpha^{N-j}(z - \alpha)^j + \alpha^N.
\]
So, for $N > m-n$,
\[
\begin{array}{ll}
s(z)z^N = z^{N+n} = & \displaystyle (z - \alpha)^m \sum_{j=0}^{N+n-m}
\begin{pmatrix}
j + m - 1\\
m-1
\end{pmatrix}
\alpha^{j-1}z^{N+n-m-j} \\
& \displaystyle + \quad \sum_{j=1}^{m-1}
\begin{pmatrix}
N+n\\
j
\end{pmatrix}
\alpha^{N+n-j}(z - \alpha)^j + \alpha^{N+n}.
\end{array}
\]
Put
\[
a_{-i} = \begin{pmatrix} i+m-1 \\ m-1 \end{pmatrix}, i=0,1,2,\dots,\quad
b_{-j} =
\left\{ \begin{array}{ll}
0 & \mbox{if $j < m - n$} \\
a_{-(j - (m-n))} & \mbox{if $j \geq m - n$}
\end{array} \right .
\]
and
\[
r(z) = \sum_{j=1}^{m-1}
\begin{pmatrix}
N+n\\
j
\end{pmatrix}
\alpha^{N+n-j}(z - \alpha)^j + \alpha^{N+n}.
\]
Then $\deg (r) < m = \deg (q)$ and
\[
\begin{array}{ll}
s(z)z^N & = q(z) \displaystyle \sum_{j=0}^{N+n-m}
\begin{pmatrix}
j + m - 1\\
m-1
\end{pmatrix}
\alpha^{j-1}z^{N+n-m-j} + r(z) \\
& = q(z ) \displaystyle \sum_{j=0}^{N+n-m}
a_{-j}z^{N+n-m-j} + r(z)\\
& = q(z)\displaystyle \sum_{j=0}^{N+n-m}
a_{-((N+n)-m-j)}z^{j} + r(z)\\
& = q(z)\displaystyle \sum_{j=0}^{N+n-m}
b_{-(N-j)}z^{j} + r(z)
\end{array}
\]
from which it follows that
\[
T_\om z^N = \sum_{j=0}^{N+n-m} b_{-(N-j)}z^{j} = \sum_{j=0}^{N}
b_{-(N-j)}z^{j}
\]
as $b_{-j} = 0$ for $j < m-n$. Thus the matrix representation of $T_\om$ is given by
\[
\begin{array}{l}
\left (
\begin{array}{cccccc}
0 \cdots 0 & a_{0} & a_{-1} & a_{-2} & a_{-3} & \cdots \\
0 \cdots 0 & 0 & a_{0} & a_{-1} & a_{-2} &  \cdots \\
0 \cdots 0 & 0 & 0 & a_{0} & a_{-1} &  \cdots \\
\vdots & \multicolumn{5}{c}{\ddots} \\
\end{array}
\right )\\
\qquad \underbrace{\qquad}_{m - n}
\end{array}.
\]
Since
\[
a_{-j} = \begin{pmatrix}
j + m - 1\\
m-1
\end{pmatrix} = \frac{(j+m-1)(j+m-2)\cdots j}{(m-1)!} \leq \frac{(j+m-1)^{m-1}}{(m-1)!}
\]
we have $a_{-j} = O(j^{m-1})$.
\end{proof}

\begin{proof}[\textbf{Proof of Theorem \ref{T:Mainthm2}}]
Let $\omega =s/q\in \Rat$ with $s,q\in\cP$ co-prime and $q$ having a root on $\BT$. By Lemma  \ref{L:reducRatT}, $\om$ can be written uniquely as $\om = \om_0 + \om_1$ with $\om_0\in\Rat_0(\BT)$ and $\om_1\in\Rat$ with no poles on $\BT$. In particular, $\om_1\in  L^\infty(\BT)$. Then $[T_\om] = [T_{\om_0}] + [T_{\om_1}]$ and $[T_{\om_0}]$ is as in Proposition \ref{P:matrix}. Moreover, $[T_{\om_1}]$ has the form as in Theorem \ref{T:Mainthm2} and the Fourier coefficients of $\om_1$ are square summable. This completes the proof.
\end{proof}

\section{Examples}\label{S:Examples}

In the final section we discuss three examples.

\begin{example}\label{E:scalar1}
Let $\om(z) = (z - 1)^{-1}\in\Rat_0(\BT)$. Then $\om = s/q$ with $s \equiv 1$ and $q(z) = z - 1$.  By Theorem \ref{T:Mainthm1a} we have
\[
\Dom(T_\om) = (z-1)H^p + \BC,
\qquad \kernel(T_\om)=\cP_{0}=\BC,
\qquad \Ran (T_\om) = H^p
\]
and so $T_\om$ is Fredholm. These facts can also be shown explicitly. By Proposition \ref{P:DomRanIncl} it suffices to establish that $\Dom(T_\om) \subset (z-1)H^p$ and so consequently $H^p \subset \Ran(T_\om)$. To this end let  $g\in\Dom(T_\om)$. Then there are $h\in H^p$ and $c\in\BC$ with $\om g = (z-1)^{-1}g = h + \frac{c}{z-1}$. Then $g = (z-1)h + c$ showing that $g\in (z-1)H^p + \BC$. For $h\in H^p$ put $g = (z-1)h$, then $g\in H^p$ with $\om g = (z-1)^{-1} (z -1)h = h$, showing that $g\in \Dom(T_\om)$ and $T_\om g = h$.

That $\kernel (T_\om) = \BC$ is also easily verified directly, as for $c\in\BC$, $\om c = 0 + \frac{c}{z-1}$. Thus $T_\om c = 0$ and so $\BC\subset \kernel(T_\om)$. The converse follows from Lemma \ref{L:domain} as for $g\in\kernel(T_\om)$, $g = c$ for some $c\in\BC$.

For the matrix representation,  note that
\[
z^n - 1 = (1 + z + z^2 + \cdots + z^{n-1} )(z - 1)
\]
or equivalently
\[
(z - 1)^{-1}z^n = 1 + z + z^2 + \cdots + z^{n-1}  + (z - 1)^{-1}
\]
and so
\[
T_\om z^n = 1 + z + z^2 + \cdots + z^{n-1}.
\]
From this it follows that the matrix representation $[T_\om]$ with respect to the standard  basis $\{ z^n\}_{n=0}^\infty$ of $H^p$ is given by
\[
[T_\om] =
\left (
\begin{array}{cccccc}
0 & 1 & 1 & 1 & 1 & \cdots \\
0 & 0 & 1 & 1 & 1 & \cdots \\
0 & 0 & 0 & 1 & 1 & \cdots \\
& \vdots &  & \ddots &  &  \cdots\\
\end{array}
\right) .
\]

Let $[T_2]$ be given by
\[
[T_2]=
\left (
\begin{array}{cccccc}
0 & 0 & 0 & 0 & 0 & \cdots \\
1 & -1 & 0 & 0 & 0 & \cdots \\
0 & 1 & -1 & 0 & 0 & \cdots \\
0 & 0 & 1 & -1 & 0 & \cdots \\
& \vdots &  & \ddots &  &  \cdots\\
\end{array}
\right ) .
\]
This is
the matrix representation of $T_2 = S + P_1 - I$ where $S=T_z$ is the forward shift operator, $P_1$ the projection onto the first component and $I$ the identity operator on $H^p$. Then $T_2$ is a generalised inverse of $T_\om$ and a right-sided inverse of $T_\om$.
\end{example}

\begin{example}\label{E:scalar2}
Let $\omega (z ) = \displaystyle\frac{z - \alpha}{z - 1}\in\Rat(\BT)$ for $\alpha\in\mathbb{C}$, $\al\neq 1$. From Lemma \ref{L:domain}
\[
\Dom (T_\om) = \left\{ g\in H^p : \om g = h + \frac{c}{z - 1}, h\in H^p, c\in\BC \right\} .
\]
For $\alpha\not\in\BT$ by Theorem \ref{T:DomRanIds},
\[
\Dom(T_\om) = (z-1)H^p+\BC, \qquad \Ran(T_\om) = (z-\alpha)H^p+\wtil{P} = (z-\alpha) H^p+\BC ,
\]
since $\wtil{P} = \{c\in\BC : c(z-1) = c_1(z-\alpha) + c_2,\quad c_1,c_2\in\BC\}=\BC$ by Lemma \ref{L:tilP}. This can be shown explicitly, and also holds for $\alpha\in\BT$. By Proposition \ref{P:DomRanIncl} it suffices to show $\Ran(T_\om) \subset (z-\alpha)H^p + \BC$. Note that for $h\in H^p$ we have $h\in \Ran(T_\om)$ if and only if there exist $g\in H^p$ and $c\in\BC$ such that
\[
\frac{z-\al}{z-1}g=h+\frac{c}{z-1},\quad\mbox{i.e.,}\quad \frac{z-1}{z-\al}h + \frac{c}{z-\al}=g\in H^p.
\]
Now use that $\frac{z - 1}{z - \alpha } = 1 + \frac{\alpha - 1}{z - \alpha}$ and $h\in H^p$, to arrive at
\begin{align*}
h\in \Ran(T_\om) \quad &\Longleftrightarrow \quad
\frac{\al-1}{z-\al}h + \frac{c}{z-\al}\in H^p,\ \
\mbox{for some $c\in\BC$}\\
&\Longleftrightarrow \quad
h + \frac{c}{\al-1}\in\frac{z-\al}{\al-1} H^p=(z-\al)H^p,\ \
\mbox{for some $c\in\BC$}\\
&\Longleftrightarrow \quad
h\in(z-\al)H^p+ \BC.
\end{align*}
So $\Ran (T_\om) = (z - \alpha)H^p + \BC$.

For the matrix representation with respect to the basis $\{z^n\}_{n=0}^\infty$, note that $\displaystyle \frac{z-\alpha}{z-1} = 1 + \frac{1-\alpha}{z-1} = 1 + c(z-1)^{-1}$ where $c = 1 - \alpha$. Then the matrix representation with respect to the standard basis $\{z^n\}_{n=0}^\infty$ of $H^p$ is given by
\[
[T_\om] =
\left (
\begin{array}{cccccc}
1 & c & c & c & c & \cdots \\
0 & 1 & c & c & c & \cdots \\
0 & 0 & 1 & c & c & \cdots \\
& \vdots &  & \ddots &  &  \cdots\\
\end{array}
\right).
\]

From Theorem \ref{T:Mainthm1}, $T_\om$ is Fredholm if and only if $\alpha\not\in\BT$, which also follows from the fact that $\Ran (T_\om) = (z - \alpha)H^p + \BC$. From Theorem \ref{T:index} $T_\om$ is invertible for $\alpha\in\BD$ and by Lemma \ref{L:kernel2},
\[
\kernel(T_\om) = \{ c/(z - \alpha) \colon c\in\BC\}
\]
in case $\alpha\not\in\BD$.

For $\alpha\in\BD$, let $T^+$ be the operator on $H^p$ with the matrix representation with respect to the standard basis $\{z^n\}_{n=0}^\infty$ of $H^p$ be given by
\[
[T^+] = \left (
\begin{array}{cccccc}
1 & -c & -c\alpha & -c\alpha^2 & -c\alpha^3 & \cdots \\
0 & 1 & -c & -c\alpha & -c\alpha^2 & \cdots \\
0 & 0 & 1 & -c & -c\alpha & \cdots \\
& \vdots &  & \ddots &  &  \cdots\\
\end{array}
\right)
\]
Then $T^+$ is the bounded inverse for $T_\om$.

For $\alpha\not\in\BD$, $T_\om$ is surjective and the operator $T^\sharp$ with the matrix representation with respect to the standard basis of $H^p$ given by
\[
[T^\sharp] = \alpha^{-1}\left (
\begin{array}{cccccc}
1 & 0 & 0 & 0 & 0 & \cdots \\
c\alpha^{-1} & 1 & 0 & 0 & 0 & \cdots \\
c(\alpha)^{-2} & c\alpha^{-1} & 1 & 0 & 0 & \cdots \\
& \vdots &  & \ddots &  &  \cdots\\
\end{array}
\right)
\]
is a right-sided inverse for $T_\om$.
\end{example}

\begin{example}\label{E:scalar3}
Let $\om (z ) = \frac{z + 1}{(z - 1)^2}\in \Rat_0(\BT)$. From Lemma \ref{L:domain}
$$\Dom (T_\om) = \left\{ g\in H^p : \om g = h + \frac{r}{(z - 1)^2}, h\in H^p, \deg (r)\leq 1 \right\}.$$
From Proposition \ref{P:DomRanIncl} we see that $(z - 1)^2H^p + \cP_1 \subset \Dom (T_\om ) $ and $(z + 1)H^p + \mathbb{C} \subset \Ran (T_\om )$. In this case $\om$ has a zero on $\BT$, so $T_\om$ is not Fredholm. Nonetheless, we will show that $(z + 1)H^p + \mathbb{C} = \Ran (T_\om )$ and $(z - 1)^2H^p + \cP_1 = \Dom (T_\om )$.

To this end, let $h\in\Ran(T_\om)$. Then there exists an $f\in H^p$ and a polynomial $r$ with $\deg (r)\leq 1$ with
\[
f = \frac{(z - 1)^2}{z + 1}h + \frac{r}{z + 1} \in H^p.
\]
Note that $\frac{(z - 1)^2}{z + 1} = z - 3 + \frac{4}{z + 1}$ and $\frac{r}{z + 1} = c_1 +  \frac{c_2}{z + 1}$ for $c_1,c_2\in\BC$. So we have
\begin{align*}
h\in\Ran (T_\omega)
& \Longleftrightarrow
(z - 3) h + \frac{4}{z + 1} h +c_1+ \frac{c_2}{z + 1}\in H^p,\ \mbox{for some $c_1,c_2\in\BC$} \\
& \Longleftrightarrow
\frac{4}{z + 1} h + \frac{c_2}{z + 1}\in H^p,\ \mbox{for some $c_2\in\BC$} \\
& \Longleftrightarrow
 h + \frac{c_2}{4}\in \frac{z+1}{4} H^p=(z+1)H^p,\ \mbox{for some $c_2\in\BC$} \\
& \Longleftrightarrow
 h \in (z+1)H^p+\BC.
\end{align*}
So $\Ran (T_\omega )= (z + 1)H^p + \mathbb{C}$. From Theorem \ref{T:Mainthm1} it follows that $T_\om$ is not Fredholm, which can also be seen directly from the fact that $\Ran(T_\om) = (z+1)H^p + \BC$ which is not closed.

To show $\Dom(T_\om) = (z-1)^2H^p + \BP_1$, let $f\in\Dom (T_\om)$ then there are $h\in H^p$ and $az + b = r \in\cP_1$ with
\[
\om f = h + \frac{az + b}{(z - 1 )^2},\quad\mbox{or equivalently,}\quad
f = \frac{(z - 1)^2}{z + 1}h + \frac{az + b}{z + 1}.
\]
Since $h\in\Ran (T_\om) = (z + 1)H^p + \BC$ there are $g\in H^p$ and $c\in\BC$ with $h = (z+1) g + c$. As $\frac{(z - 1)^2}{z + 1} =  z - 3 + \frac{4}{z + 1}$ and $\frac{az + b}{z + 1} = a + \frac{b - a}{z + 1}$, we find that
\begin{align*}
f
& = \displaystyle\frac{(z - 1)^2}{z + 1}\left((z + 1)g + c \right ) + \frac{az + b}{z + 1}\\
& = \displaystyle(z - 1)^2 g + (z - 3)c + \frac{4c}{z + 1} + a + \frac{b-a}{z + 1}.
\end{align*}
This implies that
\[
\frac{4c + b - a}{z + 1} = f - (z - 1)^2g - (cz - 3c + a) \in H^p,
\]
which, by Lemma \ref{L:h2_rat}, can only happen if $\frac{4c + b - a}{z + 1}=0$.

As $\frac{4c + b - a}{z + 1} = f - (z - 1)^2g - (cz - 3c + a) \in H^p$ from Lemma \ref{L:h2_rat} we get $\frac{4c + b - a}{z + 1} = 0$. Thus $f = (z - 1)^2g + (cz - 3c + a) \in (z - 1)^2 H^p + \cP_1$ and so $\Dom (T_\om) = (z - 1)^2 H^p + \cP_1$.

From Lemma \ref{L:kernel2} it follows that $\kernel(T_\om) = \BC$ and for the matrix representation with respect to $\{z^n\}_{n=0}^\infty$ note

\[
\frac{z + 1}{(z - 1)^2}z^k = \sum_{n=0}^{k-1} (2n + 1)z^{k-n-1} + \frac{(2k + 3)z - (2k + 1)}{(z - 1)^2}
\]
and so the matrix representation $[T_\om]$ given by
\[[T_\om] =
\left (
\begin{array}{cccccc}
0 & 1 & 3 & 5 & 7 & \cdots \\
0 & 0 & 1 & 3 & 5 & \cdots \\
0 & 0 & 0 & 1 & 3 & \cdots \\
& \vdots &  & \ddots &  &  \cdots\\
\end{array}
\right) .
\]

\end{example}

\paragraph{\bf Acknowledgement}
The present work is based on research supported in part by the National Research Foundation of South Africa. Any opinion, finding and conclusion or recommendation expressed in this material is that of the authors and the NRF does not accept any liability in this regard.

We would also like to thank the anonymous referee for the useful suggestions, corrections and comments.

\end{document}